\documentclass[12pt, ams fonts]{amsart}

\usepackage{color}
\usepackage{stmaryrd}
\usepackage{tikz-cd}
\usepackage{mathtools}
\usepackage{amssymb,amsmath}
\usepackage{amsfonts}
\usepackage{graphicx}

\usepackage[vcentermath,enableskew]{youngtab}

\input xy
\xyoption{all}

\font\teneufm=eufm10 \font\seveneufm=eufm7 \font\fiveeufm=eufm5
\newfam\eufmfam
\textfont\eufmfam=\teneufm \scriptfont\eufmfam=\seveneufm
\scriptscriptfont\eufmfam=\fiveeufm

\newcommand{\C}{\mathbb{C}}

\newcommand{\Z}{\mathbb{Z}}

\newcommand{\np}{\medskip\noindent}

\DeclareMathOperator{\Span}{Span} 
\DeclareMathOperator{\Hom}{Hom}

\DeclareMathOperator{\Mat}{Mat}
\DeclareMathOperator{\GL}{GL}

\DeclareMathOperator{\I}{I}

\DeclareMathOperator{\End}{End}

\numberwithin{equation}{section}

\newtheorem{definition}{Definition}[section]

\theoremstyle{remark}

\newtheorem{remark}[definition]{Remark}

\theoremstyle{plain} 

\newtheorem{theorem}[definition]{Theorem}
\newtheorem{lemma}[definition]{Lemma}
\newtheorem{corollary}[definition]{Corollary}
\newtheorem{proposition}[definition]{Proposition}

\def\Z{\mathbb Z}
\def\C{\mathbb C}

\begin{document}

\title{On $\mathcal{U}(\mathfrak{h})$-free modules over $\mathfrak{sl} (m|n)$}

\author[I. Dimitrov]{Ivan Dimitrov}
	\address{I. Dimitrov: Department of Mathematics and Statistics, Queen's University, Kingston, ON K7L 3N6, Canada}
	\email{dimitrov@queensu.ca}	
\author[K. Nguyen]{Khoa Nguyen}
	\address{K. Nguyen: Department of Mathematics and Statistics, Queen's University, Kingston, ON K7L 3N6, Canada}
	\email{k.nguyen@queensu.ca}

\begin{abstract}
We study two categories of $\mathcal{U}(\mathfrak h)$-free $\mathfrak{sl}(m|n)$-modules of total rank 2: $\mathcal{M}_{\mathfrak{sl}(m|n)}(2)$, whose objects are free of rank 2 over $\mathcal{U}(\mathfrak h)$ which are not necessarily $\mathbb Z_2$-graded, and $\mathcal{M}_{\mathfrak{sl}(m|n)}(1|1)$, whose objects are supermodules with even and odd parts each isomorphic to $\mathcal{U}(\mathfrak h)$. For $\mathfrak{sl}(m|1)$ we give a complete classification in both categories, and we prove that for $m,n\geq 2$ both categories are empty. 
\end{abstract}
	

\maketitle

\section{Introduction}
\noindent
Modules over Lie algebras and superalgebras are essential and have many applications in mathematics and physics. The study of modules over Lie algebras and superalgebras is usually divided into different categories. The first important category is the category of weight modules. In particular, such a category consists of modules that decompose into direct sums of their weight space with respect to a fixed Cartan subalgebra $\mathfrak{h}$. Weight modules over Lie superalgebras have been studied in \cite{DMP, G,G2,H} and the references therein. In contrast, $\mathcal{U}(\mathfrak{h})$-free modules provide important examples of non‑weight modules: these are modules that, when restricted to the Cartan subalgebra $\mathfrak{h}$, are free of finite rank. This class of modules was first introduced by J.~Nilsson \cite{Nil1} and, independently, by H.~Tan and K.~Zhao \cite{TZ2}, using a different approach. In particular, J.~Nilsson classified all $\mathcal{U}(\mathfrak{h})$-free modules of rank $1$ over $\mathfrak{sl}(n+1)$ and $\mathfrak{sp}(2n)$ \cite{Nil1,Nil2}. H.~Tan and K.~Zhao subsequently classified all $\mathcal{U}(\mathfrak{h})$-free modules of rank $1$ over the Witt algebras $W_n^+$ and $W_n$ \cite{TZ1}. Since then, additional families of finite-rank $\mathcal{U}(\mathfrak{h})$-free modules have been constructed for $\mathfrak{sl}(2)$ \cite{MP} and for $\mathfrak{sl}(n+1)$ \cite{GN}.    

\np
The categories of $\mathcal{U}(\mathfrak{h})$-free modules of rank $1$ over basic Lie superalgebras were first investigated by Y.~Cai and K.~Zhao \cite{CZ}. They proved that these categories are empty for all basic Lie superalgebras except $\mathfrak{osp}(1|2n)$, and they classified all $\mathcal{U}(\mathfrak{h})$-free modules of rank $1$ over $\mathfrak{osp}(1|2n)$. Since then, various families of $\mathcal{U}(\mathfrak{h})$-free modules have been constructed and studied over Lie superalgebras beyond the basic ones—for example, the super-Virasoro algebras \cite{YYX1}, the untwisted $N=2$ superconformal algebras \cite{YYX2}, the twisted $N=2$ superconformal algebra \cite{CDL}, and the topological $N=2$ super-$\mathrm{BMS}_3$ algebra \cite{LSZ}. We note that the classification in \cite{CZ} includes modules that are not $\mathbb{Z}_2$-graded: in particular, certain $\mathcal{U}(\mathfrak{osp}(1|2n))$-modules are treated purely as ungraded objects. Motivated by this, we continue in the same vein and study higher-rank $\mathcal{U}(\mathfrak{h})$-free modules over basic Lie superalgebras both in the ungraded setting and, when available, with their natural $\mathbb{Z}_2$-grading. Accordingly, these modules are treated as objects in the categories $\mathcal{M}_{\mathfrak{sl}(m|n)}(k)$ and $\mathcal{M}_{\mathfrak{sl}(m|n)}\bigl(k'|k''\bigr)$, where $k, k',k''\in \Z_{\geq 1}$. 

\np
The content of the paper is organized as follows. Section 2 collects the necessary background on $\mathfrak{sl}(m|n)$ and gives definitions of the categories $\mathcal{M}_{\mathfrak{sl}(m|n)}\bigl(k\bigr)$, $\mathcal{M}_{\mathfrak{sl}(m|n)}\bigl(k'|k''\bigr)$, and $\mathcal{M}^0_{\mathfrak{sl}(m|n)}\bigl(k'|k''\bigr)$ . Section 3 analyzes objects in $\mathcal{M}_{\mathfrak{sl}(1|1)}(2)$ and their relationship with modules over a string algebra. In particular, we show that $\mathcal{M}_{\mathfrak{sl}(1|1)}(2)$ contains exactly two isomorphism classes of objects, both of which have infinite length (Theorem \ref{mainthmsl(1|1)}). Section~4 provides a complete classification of objects in $\mathcal{M}_{\mathfrak{sl}(m|1)}(2)$, $\mathcal{M}_{\mathfrak{sl}(m|1)}(1|1)$, and
$\mathcal{M}^{0}_{\mathfrak{sl}(m|1)}(1|1)$. In particular, the isomorphism classes of objects are parameterized by ${\bf a} \in (\mathbb{C}^\times)^m$ and a subset $\mathcal{S} \subset \{1, \dots, m\}$ (Theorems \ref{classificationsl(m|1)}, Theorem \ref{classificationsl(m|1)-(1|1)-part1}, and Proposition \ref{classificationsl(m|1)-(1|1)-paritypres}). For all integers $m,n\geq 2$, Section 5 establishes that $\mathcal{M}_{\mathfrak{sl}(m|n)}(2)$ and $\mathcal{M}_{\mathfrak{sl}(m|n)}(1|1)$ are empty. 

\section{Preliminaries and Notations}
\noindent
Throughout the paper, we write $\mathbb{Z}$, $\mathbb{C}$, and $\mathbb{C}^\times$ for the sets of integers, complex numbers, and nonzero complex numbers, respectively. For $k\in\mathbb{Z}$, set
$$
\mathbb{Z}_{\geq k} := \{\, i\in\mathbb{Z} \mid i\geq k \,\}.
$$
For a Lie (super)algebra $\mathfrak{g}$, we denote its universal enveloping algebra by $\mathcal{U}(\mathfrak{g})$ and fix a Cartan subalgebra $\mathfrak{h}$. For a ring $R$, let $R^\times$ be its group of units and $\Mat_N(R)$ the ring of $N\times N$ matrices with entries in $R$. We denote by $e_i \in R^N$ the column vector with a $1$ in the $i$-th position and $0$ elsewhere. All vector spaces
(including Lie superalgebras, universal enveloping algebras, etc.) are defined over $\mathbb{C}$.
The dimension of a $\Z_2$-graded vector space $V = V_{\bar{0}} \oplus V_{\bar{1}}$ is $(\dim V_{\bar{0}}|\dim V_{\bar{1}})$.

\subsection{Basis of  $\mathfrak{sl} (m|n)$}\label{slbasis}
 We denote ${\bf{m}} := \{1, \dotsc, m\}$ and ${\bf{\bar{n}}} := \{\bar{1}, \dotsc, \bar{n}\}$. 
 For $I, J \in {\bf{m}}\, \cup \, {\bf{\bar{n}}}$, we denote by $e_{IJ}$ the $(m+n) \times (m+n)$ matrix 
 with zeros everywhere except for a $1$ in the $(I,J)$-position. 
 Set 
\begin{equation} \label{basis-h}
h_\iota := \left\{\begin{array}{ccl} e_{ii} + e_{\bar n \bar n} & \text{if} & \iota = i \in {\bf m} ,\\
e_{\bar{j}\bar{j}}+e_{mm} & \text{if} & \iota = \bar{j} \in {\bf{\bar{n}}}\setminus\{\bar{n}\} .
\end{array} \right. 
\end{equation}
The set 
$$
\{\, h_\iota,\; e_{IJ}\,|\, \iota \in \mathbf m \cup \left({\bf{\bar{n}}}\setminus\{\bar{n}\} \right)  ,\ I,J\in {\bf m}\cup{\bf{\bar{n}}},\ I\neq J \}
$$
is a basis of $\mathfrak{sl}(m|n)$. Furthermore, 
$$
\mathfrak{h} := \Span\{\, h_{\iota}
 \bigm| \iota\in \mathbf m \cup \left({\bf{\bar{n}}}\setminus\{\bar{n}\} \right) \}
$$
is a Cartan subalgebra of $\mathfrak{sl}(m|n)$. For the rest of the paper we fix the basis of $\mathfrak{sl}(m|n)$ and the Cartan subalgebra $\mathfrak{h}$ described above.

\np
Since $\mathfrak{h}$ is abelian, $\mathcal U(\mathfrak{h}) = \mathcal S(\mathfrak{h}) = \mathcal P(\mathfrak{h}^*)$ where $\mathcal{S}(-)$ and 
$\mathcal{P}(-)$ denote respectively the symmetric and polynomial algebras on a vector space.
Identifying the basis \eqref{basis-h} with its dual basis of $\mathfrak{h}^*$ and setting $\mathbf{h} := (h_1, \dots,h_m,h_{\bar{1}},\dots, h_{\overline{n-1}})$,
we identify $\mathcal{U}(\mathfrak h)$ with $\mathbb{C}[\mathbf h]$.
For $\iota \in \mathbf m \cup \left({\bf{\bar{n}}}\setminus\{\bar{n}\} \right)$, we define the automorphism $\sigma_{\iota}$ of $\mathbb{C}[{\bf h}]$ by
$$\sigma_{\iota}(h_{\beta}) := \left\{ \begin{array}{ccl}h_{\beta} -1 &\text{if} & \beta = \iota \ , \\
h_{\beta} &\text{if} & \beta \neq \iota \ . 
\end{array} \right.$$
Since $\sigma_{\iota}$ shifts the variable $h_{\iota}$ by 1, $\sigma_1, \dots, \sigma_m, \sigma_{\bar{1}}, \dots, \sigma_{\overline{n-1}}$
are pairwise commuting automorphisms. Set
$$
\Delta \;:=\; \prod_{i\in \mathbf m} \sigma_i \;=\; \sigma_1\dots\sigma_m,
\qquad \overline\Delta:= \prod_{\bar{j}\in \left({\bf{\bar{n}}}\setminus\{\bar{n}\} \right)} \sigma_{\bar{j}} = \sigma_{\bar{1}}\dots \sigma_{\overline{n-1}}.
$$

\np
Until Section \ref{section5} we study the superalgebras $\mathfrak{sl}(m|1)$ and hence 
only $\sigma_1, \dots, \sigma_m$ and $\Delta$ are considered.

\subsection{The categories $\mathcal{M}_{\mathfrak{sl}(m|n)}\bigl(k\bigr)$ and $\mathcal{M}_{\mathfrak{sl}(m|n)}\bigl(k'|k''\bigr)$}
If $\mathfrak{g}$ is a Lie superalgebra, its universal enveloping algebra $\mathcal{U}(\mathfrak{g})$ is an associative superalgebra.
Forgetting the $\Z_2$-grading on $\mathcal{U}(\mathfrak{g})$, we may view it as an associative algebra. Consequently, 
two different categories of $\mathcal{U}(\mathfrak{g})$-modules arise: $\Z_2$-graded and non-graded. These correspond 
to $\mathfrak{g}$-modules which are $\Z_2$-graded or non-graded respectively. 

\np
Let $k,k',k'' \in \mathbb{Z}_{\geq 1}$. Define $\mathcal{M}_{\mathfrak{sl}(m|n)}(k)$ and 
$\mathcal{M}_{\mathfrak{sl}(m|n)}(k'|k'')$ to be categories of $\mathcal{U}(\mathfrak{sl}(m|n))$-modules 
whose restrictions to 
$\mathcal{U}(\mathfrak h)$ are free of ranks $k$ and $(k'|k'')$ respectively. Note
that the objects of $\mathcal{M}_{\mathfrak{sl}(m|n)}(k'|k'')$ are $\Z_2$-graded while
those of $\mathcal{M}_{\mathfrak{sl}(m|n)}(k)$ are not. Respectively, the morphisms of 
$\mathcal{M}_{\mathfrak{sl}(m|n)}(k'|k'')$ are $\Z_2$-graded and those of $\mathcal{M}_{\mathfrak{sl}(m|n)}(k)$ are not.
Finally, define $\mathcal M^{0}_{\mathfrak{sl}(m|n)}(k'| k'')$ to be the category whose objects are the objects of
$\mathcal{M}_{\mathfrak{sl}(m|n)}(k'|k'')$ and in which only even homomorphisms are considered, i.e.,
$$
\Hom_{\mathcal{M}^{0}_{\mathfrak{sl}(m|n)}(k'|k'')}(U,V) =
 \bigl(\Hom_{\mathcal{M}_{\mathfrak{sl}(m|n)}(k'|k'')}(U,V)\bigr)_{\bar 0} \ 
$$
for any two objects $U$ and $V$ in $\mathcal{M}_{\mathfrak{sl}(m|n)}(k'|k'')$.

\begin{remark} \label{identificationofM}
For each $M\in \mathcal{M}_{\mathfrak{sl}(m|n)}(k)$ (resp. $M\in \mathcal{M}_{\mathfrak{sl}(m|n)}(k'|k'')$), we identify 
the $\mathcal{U}(\mathfrak{sl}(m|n))$-module $M$ with $\C[{\bf h}]^{\oplus k}$ (respectively, 
$\C[{\bf h}]^{\oplus k'} \oplus \C[{\bf h}]^{\oplus k''})$, where
the two summands are the even and the odd parts of $M$). In particular, 
for any $f_1({\bf h}),\dots,f_N({\bf h}) \in \C[\mathbf h]$ and $\iota \in \kappa$,
$$h_{\iota} \; \cdot \; \bigl[f_1({\bf h})\;\;\dots\;\;f_{N}({\bf h})\bigr]^{\mathsf T} 
\;=\; \bigl[h_{\iota}f_1({\bf h})\;\;\dots\;\;h_{\iota}f_{N}({\bf h})\bigr]^{\mathsf T}. $$
\end{remark}

\np
Cai and Zhao \cite[Proposition 2.6]{CZ} proved that, if $\mathfrak{g}$ a basic Lie superalgebra 
not isomorphic to $\mathfrak{osp}(1|2n)$, then $\mathcal{M}_{\mathfrak{g}}(1) = \emptyset$.
For $\mathfrak{g} \simeq \mathfrak{osp}(1|2n)$ they classified the objects of $\mathcal{M}_{\mathfrak{g}}(1)$.
Naturally, none of these modules admits a $\Z_2$-grading. In this paper we study the categories of modules of rank 2, 
i.e., the categories $\mathcal{M}_{\mathfrak{g}}(2)$ and $\mathcal{M}_{\mathfrak{g}}(1|1)$ when 
$\mathfrak{g} \simeq \mathfrak{sl}(m|n)$.

\section{The category $\mathcal{M}_{\mathfrak{sl} (1|1)}(2)$}
\noindent
The main goal of this section is to study the category $\mathcal{M}_{\mathfrak{sl}(1|1)}(2)$ and to classify all of its objects up to isomorphism. For convenience, we fix a basis $\{x,y,h\}$ of $\mathfrak{sl}(1|1)$, where
$$
x := e_{1\bar{1}}, \quad y := e_{\bar{1}1}, \quad h := e_{11} + e_{\bar{1}\bar{1}}.
$$ 

\subsection{The description of $\mathcal{M}_{\mathfrak{sl} (1|1)}(2)$}
 For each $M \in \mathcal{M}_{\mathfrak{sl}(1|1)}(2)$, $M = \mathbb{C}[h]^{\oplus 2},$
and the $\mathcal{U}(\mathfrak{sl}(1|1))$–module structure is as stated in the lemma below.

\begin{proposition}\label{structure_sl(1|1)}
\begin{enumerate}
\item[(i)]
Let $P,Q \in \Mat_2\bigl(\mathbb{C}[h]\bigr)$ satisfy
\begin{equation} \label{eq-sl(1|1)}
    P^2 = Q^2 = 0, \qquad PQ + QP = h \I_2.
\end{equation}
Define $M(P,Q):=\C[h]^{\oplus 2}$ with the following $\mathcal{U}(\mathfrak{sl}(1|1))$–action:
$$
x \cdot \begin{bmatrix} f_1(h) \\ f_2(h) \end{bmatrix} = P \begin{bmatrix} f_1(h) \\ f_2(h) \end{bmatrix},
\qquad
y \cdot \begin{bmatrix} f_1(h) \\ f_2(h) \end{bmatrix} = Q \begin{bmatrix} f_1(h) \\ f_2(h) \end{bmatrix}.
$$
Then $M(P,Q)$ is an $\mathcal{U}(\mathfrak{sl}(1|1))$–module. 
\item[]
\item[(ii)] If $M \in \mathcal{M}_{\mathfrak{sl}(1|1)}(2)$, then $M \simeq M(P,Q)$ for some $P,Q$ satisfying \eqref{eq-sl(1|1)}.
\item[]
\item[(iii)] $M(P,Q) \simeq M(P',Q')$ if and only if there exists $W \in \GL_2(\mathbb{C}[h])$ such that
$$
P' = W^{-1} P\, W,\qquad Q' = W^{-1} Q\, W.
$$
\end{enumerate}
\end{proposition}

\begin{proof}
\begin{itemize}
    \item[(i)]
The fact that $M(P,Q)$ is an $\mathcal{U}(\mathfrak{sl}(1|1))$-module follows from the commutation relations in
$\mathfrak{sl}(1|1)$ and \eqref{eq-sl(1|1)}.

\item[]

\item[(ii)]
If $M \in \mathcal{M}_{\mathfrak{sl}(1|1)}(2)$, set
$P:=\begin{bmatrix} x\cdot e_1 & x\cdot e_2 \end{bmatrix}$
and $Q := \begin{bmatrix} y \cdot e_1 & y \cdot e_2 \end{bmatrix}$.
Then $P$ and $Q$ satisfy \eqref{eq-sl(1|1)} because $x^2 = y^2 = 0$ and $h = [x,y] = xy + yx$ in 
$\mathcal{U}(\mathfrak{sl}(1|1))$.
Since $x h = h x$, we have $x\,f(h)=f(h)\,x$ for all $f(h)\in\C[h]$. Therefore,
$$
x \cdot \begin{bmatrix} f_1(h) \\ f_2(h) \end{bmatrix}
= f_1(h) (x \cdot e_1) + f_2(h)(x \cdot e_2)
= \begin{bmatrix} x \cdot e_1 & x \cdot e_2 \end{bmatrix}
\begin{bmatrix} f_1(h) \\ f_2(h) \end{bmatrix} = P \begin{bmatrix} f_1(h) \\ f_2(h) \end{bmatrix}.
$$
Similarly,
$$
y \cdot \begin{bmatrix} f_1(h) \\ f_2(h) \end{bmatrix}
= Q \begin{bmatrix} f_1(h) \\ f_2(h) \end{bmatrix},
$$
proving that $M \simeq M(P,Q)$.

\item[]

\item[(iii)] This follows from the fact that $[h,x] = [h,y] = 0$. \qedhere
\end{itemize}
\end{proof}

\np
\begin{theorem} \label{mainthmsl(1|1)}
Let $M \in \mathcal{M}_{\mathfrak{sl}(1|1)}(2)$. Then either
$$
M \simeq M\left( \begin{bmatrix} 0 & 1 \\ 0 & 0 \end{bmatrix}, \begin{bmatrix} 0 & 0 \\ h & 0 \end{bmatrix} \right),
\quad \text{or} \quad
M \simeq M\left( \begin{bmatrix} 0 & h \\ 0 & 0 \end{bmatrix}, \begin{bmatrix} 0 & 0 \\ 1 & 0 \end{bmatrix} \right).
$$
Moreover, 
$$
M\left( \begin{bmatrix} 0 & 1 \\ 0 & 0 \end{bmatrix}, \begin{bmatrix} 0 & 0 \\ h & 0 \end{bmatrix} \right) \not \simeq
 M\left( \begin{bmatrix} 0 & h \\ 0 & 0 \end{bmatrix}, \begin{bmatrix} 0 & 0 \\ 1 & 0 \end{bmatrix} \right).
$$
\end{theorem}

\begin{proof} 
Let $M \in \mathcal{M}_{\mathfrak{sl} (1|1)}(2)$. By Proposition \ref{structure_sl(1|1)}(i), $M \simeq M(P,Q)$.
Since $P^2 =0$, $P$ is conjugate to a 
nilpotent Jordan block 
$\begin{bmatrix}
0 & f(h)\\
0 & 0
\end{bmatrix},$ where $f(h)\in\C[h]$. From  $PQ + QP \;=\; h\,\I_2$, we see that $P \neq 0$ and thus 
$f(h) \neq 0$. By Proposition \ref{structure_sl(1|1)}(iii), we may assume that $P=\begin{bmatrix}
0 & f(h)\\
0 & 0
\end{bmatrix}$ with $f(h)\in\C[h]\setminus\{0\}$. Since $Q^{2}=0$, the matrix $Q$ can be written as
$$Q=\begin{bmatrix} q_{1}(h) & q_{2}(h)\\ q_{3}(h) & -q_{1}(h) \end{bmatrix},\ \ q_i(h)\in\C[h],\ \ \text{and}\ \ q_{2}(h)\,q_{3}(h)=-q_{1}(h)^{2}.$$

\np
Moreover,  the relation $PQ+QP=h\,\I_{2}$ forces $f(h)\,q_{3}(h)=h$.

\np
\textbf{Case 1:} Suppose $f(h)=\alpha \in\mathbb{C}^\times$, so $q_3(h)=\alpha^{-1}h$. Then
$$
P(h)=\begin{bmatrix}0 & \alpha\\0 & 0\end{bmatrix},\qquad Q(h)=\begin{bmatrix}h\,v(h) & -\alpha\,h\,v(h)^2\\\alpha^{-1}h & -h\,v(h)\end{bmatrix},\qquad v(h)\in\C[h].
$$

\np
Let $W:=\begin{bmatrix} 1 & v(h) \\[2pt] 0 & \alpha^{-1} \end{bmatrix} \in \GL_2(\C[h])$. A direct computation yields
$$
W^{-1}
\begin{bmatrix} h v(h) & -\alpha h v(h)^2 \\[2pt] \alpha^{-1} h & - h v(h) \end{bmatrix}
W
=
\begin{bmatrix} 0 & 0 \\ h & 0 \end{bmatrix},
\qquad
W^{-1}
\begin{bmatrix} 0 & \alpha \\ 0 & 0 \end{bmatrix}
W
=
\begin{bmatrix} 0 & 1 \\ 0 & 0 \end{bmatrix}.
$$

\np
Therefore,
$$
M\left(
\begin{bmatrix} 0 & \alpha \\ 0 & 0 \end{bmatrix},
\begin{bmatrix} h v(h) & -\alpha h v(h)^2 \\ \alpha^{-1} h & - h v(h) \end{bmatrix}
\right)\;\simeq\;M\left(\begin{bmatrix} 0 & 1 \\ 0 & 0 \end{bmatrix},\begin{bmatrix} 0 & 0 \\ h & 0 \end{bmatrix}
\right).
$$

\np
\textbf{Case 2:} Suppose $f(h) = \alpha h$ with $\alpha\in\mathbb{C}^\times$, so $q_3(h) = \alpha^{-1}$. Then
$$
P(h) = \begin{bmatrix} 0 & \alpha h \\ 0 & 0 \end{bmatrix}, \quad Q(h) = \begin{bmatrix} v(h) & -\alpha v(h)^2 \\ \alpha^{-1} & -v(h) \end{bmatrix},\qquad v(h)\in\C[h].
$$
Since
$$
W^{-1}
\begin{bmatrix} v(h) & -\alpha v(h)^2 \\ \frac{1}{\alpha} & -v(h) \end{bmatrix}
W = \begin{bmatrix} 0 & 0 \\ 1 & 0 \end{bmatrix},\qquad
W^{-1}\begin{bmatrix} 0 & \alpha h \\ 0 & 0 \end{bmatrix} W = \begin{bmatrix} 0 & h \\ 0 & 0 \end{bmatrix},
$$

\np
we deduce that
$$
M\left( \begin{bmatrix} 0 & \alpha h \\ 0 & 0 \end{bmatrix},
         \begin{bmatrix} v(h) & -\alpha v(h)^2 \\ \alpha^{-1} & -v(h) \end{bmatrix}
\right)
\;\simeq\;
M\left( \begin{bmatrix} 0 & h \\ 0 & 0 \end{bmatrix},
         \begin{bmatrix} 0 & 0 \\ 1 & 0 \end{bmatrix}
\right).
$$

\np
Finally, the pairs
$$
\left(
\begin{bmatrix} 0 & h \\ 0 & 0 \end{bmatrix},
\begin{bmatrix} 0 & 0 \\ 1 & 0 \end{bmatrix}
\right)
\quad\text{and}\quad
\left(
\begin{bmatrix} 0 & 1 \\ 0 & 0 \end{bmatrix},
\begin{bmatrix} 0 & 0 \\ h & 0 \end{bmatrix}
\right)
$$
lie in distinct $\GL_2(\C[h])$-orbits under conjugation, proving that
$$
M\left( \begin{bmatrix} 0 & 1 \\ 0 & 0 \end{bmatrix}, \begin{bmatrix} 0 & 0 \\ h & 0 \end{bmatrix} \right) \not \simeq
 M\left( \begin{bmatrix} 0 & h \\ 0 & 0 \end{bmatrix}, \begin{bmatrix} 0 & 0 \\ 1 & 0 \end{bmatrix} \right). \hfill \qedhere
$$
\end{proof}

\np 
\begin{proposition}
Every $M\in\mathcal{M}_{\mathfrak{sl}(1|1)}(2)$ has infinite length. More precisely,
each submodule of $M=M\!\left(\begin{bmatrix}0&1\\0&0\end{bmatrix},\begin{bmatrix}0&0\\ h&0\end{bmatrix}\right)$
(respectively, $M\left( \begin{bmatrix} 0 & h \\ 0 & 0 \end{bmatrix}, \begin{bmatrix} 0 & 0 \\ 1 & 0 \end{bmatrix} \right)$)
equals $J\oplus J$ or $J \oplus h J$ (respectively, $J \oplus J$ or $hJ \oplus J$) for some ideal $J \subset \C[h]$.
\end{proposition}

\begin{proof}
 Let $N$ be a submodule of $M\left( \begin{bmatrix} 0 & 1 \\ 0 & 0 \end{bmatrix}, \begin{bmatrix} 0 & 0 \\ h & 0 \end{bmatrix} \right)$. Since $\C[h]$ is a PID, $N$ is a $\mathcal U(\mathfrak h)$-free module. By
 \cite[Proposition 2.6]{CZ}, $N$ is not of rank 1. Assume $N\neq 0$, then $N \in \mathcal{M}_{\mathfrak{sl}(1|1)}(2)$.
Hence, by Theorem~\ref{mainthmsl(1|1)},
$$
N \simeq M\left( \begin{bmatrix} 0 & 1 \\ 0 & 0 \end{bmatrix}, \begin{bmatrix} 0 & 0 \\ h & 0 \end{bmatrix} \right)
\quad \text{or} \quad
N \simeq M\left( \begin{bmatrix} 0 & h \\ 0 & 0 \end{bmatrix}, \begin{bmatrix} 0 & 0 \\ 1 & 0 \end{bmatrix} \right).
$$

\np
The inclusion $N \subset M$ gives rise to a $\mathcal U(\mathfrak{sl}(1|1))$-homomorphism
$$\Phi: M\left( \begin{bmatrix} 0 & 1 \\ 0 & 0 \end{bmatrix}, \begin{bmatrix} 0 & 0 \\ h & 0 \end{bmatrix} \right) \to M \quad
\text{or}\quad 
\Phi: M\left( \begin{bmatrix} 0 & h \\ 0 & 0 \end{bmatrix}, \begin{bmatrix} 0 & 0 \\ 1 & 0 \end{bmatrix} \right) \to M.$$
 Then there exists
$$
W(h)=\begin{bmatrix} w_1(h) & w_2(h)\\ w_3(h) & w_4(h)\end{bmatrix}\in \Mat_2(\C[h])
$$
such that, for all $\mathbf f(h) \in \C[h]^{\oplus2}$, $\Phi(\mathbf f(h))=W(h)\,\mathbf f(h)$. Since $\Phi$ is a $\mathcal{U}(\mathfrak{sl}(1|1))$-homomorphism,
\begin{equation}\label{intertwiningrelations}
x\cdot\Phi(\mathbf f(h))=\Phi(x\cdot \mathbf f(h)),\qquad
y\cdot\Phi(\mathbf f(h))=\Phi(y\cdot \mathbf f(h))\quad\text{for all}\;\;\;\mathbf f(h)\in \C[h]^{\oplus2} .
\end{equation}

\np
If $N \,\simeq\, 
M\left(
\begin{bmatrix} 0 & 1 \\ 0 & 0 \end{bmatrix},
\begin{bmatrix} 0 & 0 \\ h & 0 \end{bmatrix}
\right)$, then \eqref{intertwiningrelations} yields
$$
W(h)\!\begin{bmatrix} 0 & 1 \\ 0 & 0 \end{bmatrix}
=\begin{bmatrix} 0 & 1 \\ 0 & 0 \end{bmatrix}\! W(h),
\qquad
W(h)\!\begin{bmatrix} 0 & 0 \\ h & 0 \end{bmatrix}
=\begin{bmatrix} 0 & 0 \\ h & 0 \end{bmatrix}\! W(h).
$$
\np
Therefore, $W(h) = \begin{bmatrix} f(h) & 0 \\ 0 & f(h) \end{bmatrix}$, for some $f(h) \in \mathbb{C}[h]$.

\np 
If $N \simeq M\left(\begin{bmatrix} 0 & h\\ 0&0 \end{bmatrix}, \begin{bmatrix} 0 & 0\\ 1&0 \end{bmatrix}\right),$ then \eqref{intertwiningrelations} imply
$$W(h) \begin{bmatrix} 0 & h\\ 0 & 0 \end{bmatrix} = \begin{bmatrix} 0 & 1\\ 0 & 0 \end{bmatrix} W(h), \quad W(h) \begin{bmatrix} 0 & 0\\ 1 & 0 \end{bmatrix} = \begin{bmatrix} 0 & 0\\ h & 0 \end{bmatrix} W(h).$$ 

\np
Hence, $W(h) = \begin{bmatrix} f(h) & 0 \\ 0 &hf(h) \end{bmatrix}$, for some $f(h) \in \C[h].$ 

\np 
Submodules of $M\left( \begin{bmatrix} 0 & h \\ 0 & 0 \end{bmatrix}, \begin{bmatrix} 0 & 0 \\ 1 & 0 \end{bmatrix} \right)$ are handled analogously.
\end{proof}

\subsection{Connections between objects in $\mathcal{M}_{\mathfrak{sl}(1|1)}(2)$ and modules over a string algebra}
Consider the one-vertex quiver $Q$: 
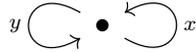
$$
\begin{tikzcd}
\bullet
  \arrow[out=-30,  in= 30,  loop, swap, "x"]
  \arrow[out=150, in=210, loop, swap, "y"]
\end{tikzcd}
$$
with two loops $x$, $y$, and let $\C Q$ be its path algebra. Then
$$ \mathcal{U}(\mathfrak{sl}(1|1)) \simeq \C Q / \rho,$$
where $\rho$ is the left ideal generated by $\{x^2, y^2\}$. 

\np
The algebra $\mathbb{C}Q / \rho$ is a \emph{string algebra} (see \cite{C}). Note that
the center of $\mathbb{C}Q / \rho$ is $\C[h]$ where $h := xy + yx$. As a consequence, the objects of $\mathcal{M}_{\mathfrak{sl}(1|1)}(2)$ can therefore be viewed as $\C Q/\rho$–modules that are free of rank~$2$ over its center.

\np
The \emph{string modules} $M_1$ and $M_2$ are defined as follows: $M_1$ and $M_2$ 
have the same underlying vector space
with a basis $\{u_i \mid i \in \mathbb{Z}_{\geq 1}\}$ and the non-zero actions of the generators $x$ and $y$ on
the basis vectors are depicted as
\begin{itemize}
    \item $M_1$:
    \begin{tikzcd}
    u_1 \arrow[r, "x"] &
    u_2 \arrow[r, "y"] &
    u_3 \arrow[r, "x"] &
    u_4 \arrow[r, "y"] &
    u_5 \arrow[r, "x"] &
    \cdots
    \end{tikzcd}
    
    \item $M_2$:
    \begin{tikzcd}
    u_1 \arrow[r, "y"] &
    u_2 \arrow[r, "x"] &
    u_3 \arrow[r, "y"] &
    u_4 \arrow[r, "x"] &
    u_5 \arrow[r, "y"] &
    \cdots \ .
    \end{tikzcd}
\end{itemize}

\begin{lemma}
The following isomorphisms hold:
    $$
    M\left( \begin{bmatrix} 0 & 1 \\ 0 & 0 \end{bmatrix}, 
             \begin{bmatrix} 0 & 0 \\ h & 0 \end{bmatrix} \right)
    \;\simeq\; M_1, \qquad
    M\left( \begin{bmatrix} 0 & h \\ 0 & 0 \end{bmatrix}, 
             \begin{bmatrix} 0 & 0 \\ 1 & 0 \end{bmatrix} \right)
    \;\simeq\; M_2.
    $$
\end{lemma}

\begin{proof}
    The maps:
\[\begin{array}{ccl}
    \Phi_1 &:& M\left( \begin{bmatrix} 0 & 1 \\ 0 & 0 \end{bmatrix}, 
             \begin{bmatrix} 0 & 0 \\ h & 0 \end{bmatrix} \right)
    \to M_1\\ &&\\
    && \begin{bmatrix} 0 \\ h^i \end{bmatrix} \mapsto u_{2i+1}, \quad \begin{bmatrix} h^i \\ 0 \end{bmatrix} \mapsto u_{2i+2}\quad\text{for all}\quad i \in \Z_{\geq 0}, \\ && \end{array} \]
\[\begin{array}{ccl}    
    \Phi_2&:& M\left( \begin{bmatrix} 0 & h \\ 0 & 0 \end{bmatrix}, 
             \begin{bmatrix} 0 & 0 \\ 1 & 0 \end{bmatrix} \right)
    \to M_2\\ &&\\
  &&  \begin{bmatrix} h^i \\ 0 \end{bmatrix} \mapsto u_{2i+1}, \quad \begin{bmatrix} 0 \\ h^i \end{bmatrix} \mapsto u_{2i+2}\quad\text{for all}\quad i \in \Z_{\geq 0}, \\ &&
    \end{array} \] 
define the desired isomorphisms.    
\end{proof}

\section{The categories $\mathcal{M}_{\mathfrak{sl} (m|1)}(2)$ and $\mathcal{M}_{\mathfrak{sl}(m|1)}\bigl(1|1\bigr)$} \label{section4}
\np
In this section we classify, up to isomorphism, all objects in the categories $\mathcal{M}_{\mathfrak{sl}(m|1)}(2)$, $\mathcal{M}_{\mathfrak{sl}(m|1)}\bigl(1|1\bigr)$, and $\mathcal{M}^0_{\mathfrak{sl}(m|1)}\bigl(1|1\bigr)$, where $m \geq 2$. Throughout, $\mathbb{C}[\mathbf h]:=\mathbb{C}[h_1,\dots,h_{m}]$. As stated in Remark~\ref{identificationofM}, we regard any object of $\mathcal{M}_{\mathfrak{sl}(m|1)}(2)$ (and likewise of $\mathcal{M}_{\mathfrak{sl}(m|1)}\bigl(1|1\bigr)$ and $\mathcal{M}^0_{\mathfrak{sl}(m|1)}\bigl(1|1\bigr)$) as the vector space $\C[{\bf h}]^{\oplus 2}$.

\begin{proposition} \label{structuresl(m|1)}
Let $M = \C[{\bf h}]^{\oplus 2} \in \mathcal{M}_{\mathfrak{sl} (m|1)}(2)$. For $I \neq J \in {\bf m}\,\cup\,\{\bar{1}\}$, set 
$$E_{IJ} := \bigl[\, e_{IJ}\!\cdot e_1 \;\;\; e_{IJ}\!\cdot e_2 \,\bigr] \in \Mat_2\bigl(\mathbb{C}[{\bf h}]\bigr).$$
Then, for any $\mathbf {f}({\bf{h}}) \in \C[\mathbf h]^{\oplus 2}$ and any $i \neq j \in \mathbf{m}$, we have 
$$\begin{array}{ccl}
e_{ij}\cdot \mathbf {f}({\bf{h}}) &=& E_{ij}({\bf h})\,\sigma_i\sigma_j^{-1}\left(\mathbf {f}({\bf{h}})\right)  ,\\
e_{i\bar 1}\cdot \mathbf {f}({\bf{h}}) &=& E_{i\bar 1}({\bf h})\,\sigma_i\Delta^{-1}\left(\mathbf {f}({\bf{h}})\right)  ,\\
e_{\bar 1 i}\cdot \mathbf {f}({\bf{h}}) & = & E_{\bar 1 i}({\bf h})\,\sigma_i^{-1}\Delta\left(\mathbf {f}({\bf{h}})\right)  .
\end{array}  $$
\end{proposition}

\begin{proof}
Since $h_ie_{i\bar{1}} = e_{i\bar{1}}h_i $, and $h_je_{i\bar{1}} - e_{i\bar{1}}h_j = -e_{i\bar{1}}$ for all $j\neq i$, it follows (by induction on the degree) that, for every $g({\bf h})\in\C[{\bf h}]$,
$$e_{i\bar{1}}\; g({\bf{h}}) = \left(\prod_{j\in{\bf{m}}\setminus \{i\}} \sigma^{-1}_j \right)\!\!\Big(g({\bf{h}})\Big)\;e_{i\bar{1}} = \Delta^{-1}_i\big(g({\bf{h}})\big)\;e_{i\bar{1}},$$
where $\Delta_i := \sigma_i^{-1} \Delta$.
Therefore,
$$e_{i\bar{1}}\cdot \begin{bmatrix} f_1({\bf{h}})  \\ f_{2}({\bf{h}})\end{bmatrix}
=\sum_{\ell=1}^{2} \Delta_i^{-1}\!\bigl(f_\ell({\bf h})\bigr)\,\bigl(e_{i\bar{1}}\cdot e_\ell\bigr)
=\Bigl[\,e_{i\bar{1}}\cdot e_1 \;\;\; \; e_{i\bar{1}}\cdot e_{2}\,\Bigr]\,
\Delta_i^{-1}\!\left( \begin{bmatrix} f_1({\bf{h}})  \\ f_{2}({\bf{h}})\end{bmatrix} \right).$$
The $\mathcal U(\mathfrak{sl}(m|1))$-actions of the remaining basis elements are obtained analogously, we omit the details.
\end{proof}

\np
Since $\left\{e_{i\bar{1}}, e_{\bar{1}i}: i\in\mathbf m \right\}$ generates $\mathfrak{sl}(m|1)$, the matrices $E_{i\bar{1}}$ and $E_{\bar{1}i}$ uniquely determines all $E_{IJ}$ and hence the module
$M\in\mathcal{M}_{\mathfrak{sl}(m|1)}(2)$. Next we study which collections of matrices $E_{i\bar{1}}, E_{\bar{1}i}$ give rise to modules
in $\mathcal{M}_{\mathfrak{sl}(m|1)}(2)$ and when the modules corresponding to two such collections of matrices are isomorphic.

\begin{remark} \label{structure(1|1)}
   If $M\in\mathcal{M}_{\mathfrak{sl}(m|1)}\bigl(1|1\bigr)$ (resp. $\mathcal{M}^0_{\mathfrak{sl}(m|1)}(1|1)$), then $M=M_{\bar0}\oplus M_{\bar1}$ with $M_{\bar\epsilon}=\C[\mathbf h]$ ($\epsilon=0,1$), endowed with the action described in Lemma~\ref{structuresl(m|1)}. Since 
    $$
x_i\cdot M_{\bar\epsilon}\subset M_{\overline{\epsilon+1}}
,\quad x_i\in \left\{e_{i\bar{1}},e_{\bar{1}i}\right\}\quad\text{while,}\quad e_{ij}\cdot M_{\bar\epsilon}\subset M_{\bar\epsilon},
$$
then
$$
E_{i\bar 1}({\bf h})=\begin{bmatrix}0 & a_{i\bar 1}({\bf h})\\b_{i\bar 1}({\bf h}) & 0\end{bmatrix},\quad
E_{\bar 1 i}({\bf h})=\begin{bmatrix}0 & a_{\bar 1 i}({\bf h})\\b_{\bar 1 i}({\bf h}) & 0 \end{bmatrix},\quad
E_{ij}({\bf h})=\begin{bmatrix} a_{ij}({\bf h}) & 0\\ 0 & b_{ij}({\bf h})\end{bmatrix}, $$
with $a_{IJ}({\bf h}),b_{IJ}({\bf h})\in \C[{\bf h}]$ for all $I,J\in {\bf m}\cup\{\bar 1\}$.

\end{remark}
\begin{definition} \label{sl(m|1)_conjugation}
Let $E_{i\bar{1}}, E_{\bar{1}i} \in \Mat_2(\mathbb{C}[{\bf{h}}])$ for each $i \in {\bf{m}}$, and let $W({\bf{h}}) \in \GL_2(\mathbb{C}[{\bf{h}}])$. For every $i\in\mathbf m$, set
$$
E'_{i\bar{1}} := W^{-1}({\bf{h}}) \, E_{i\bar{1}} \, \Delta_i^{-1}\big(W({\bf{h}})\big), \quad
E'_{\bar{1}i} := W^{-1}({\bf{h}}) \, E_{\bar{1}i} \, \Delta_i\big(W({\bf{h}})\big).
$$
 We say that the tuples $(E_{i\bar{1}},E_{\bar{1}i})_{i\in \mathbf m}$ and
$(E'_{i\bar{1}},E'_{\bar{1}i})_{i\in \mathbf m}$ are
\emph{$\mathcal{M}_{\mathfrak{sl}(m|1)}(2)$-conjugate} if there exists
$W(\mathbf h)\in \GL_2(\mathbb{C}[\mathbf h])$ for which the above identities hold.
\end{definition}

\np
\begin{remark} \label{sl(m|1)-conjugacy}
The $\mathcal{M}_{\mathfrak{sl}(m|1)}(2)$-conjugation defined above is an equivalence relation. We write
$$
\bigl(E_{i\bar{1}},E_{\bar{1}i}\bigr)_{i\in{\bf m}}
\sim_{\mathcal{M}_{\mathfrak{sl}(m|1)}(2)}
\bigl(E'_{i\bar{1}},E'_{\bar{1}i}\bigr)_{i\in{\bf m}}
$$
to indicate that the two $2m$-tuples are $\mathcal{M}_{\mathfrak{sl}(m|1)}(2)$-conjugate.
To classify all objects of $\mathcal{M}_{\mathfrak{sl}(m|1)}(2)$ it suffices to determine,
up to $\mathcal{M}_{\mathfrak{sl}(m|1)}(2)$-conjugation, all compatible tuples
$$
\bigl(E_{i\bar{1}},\,E_{\bar{1}i}\bigr)_{i\in{\bf m}}
\;\in\;
\Mat_2\bigl(\C[{\bf h}]\bigr)^{\times 2m}.
$$
\end{remark}

\subsection{Preparatory results}In this subsection, we gather the necessary results to classify objects in $\mathcal{M}_{\mathfrak{sl} (m|1)}(2)$.

\np 
\begin{lemma}\label{sl(m|1)reln}
Let $\mathcal{R}$ be a unique factorization domain, and let
$P(\mathbf h)\in \Mat_2\bigl(\mathcal{R}[\mathbf h]\bigr)$. Then the solutions $P(\mathbf h)\ \Delta^{-1}(P(\mathbf h))=0$ are the matrices of the form
$$
P(\mathbf h)=\theta(\mathbf h)\!
\begin{bmatrix}
\beta(\mathbf h)\ \Delta(\alpha(\mathbf h)) &
-\alpha(\mathbf h)\ \Delta(\alpha(\mathbf h))\\
\beta(\mathbf h)\ \Delta(\beta(\mathbf h)) &
-\alpha(\mathbf h)\ \Delta(\beta(\mathbf h))
\end{bmatrix},
$$
where $\alpha(\mathbf h),\beta(\mathbf h),\theta(\mathbf h)\in \mathcal{R}[\mathbf h]$ such that $\gcd\bigl(\alpha(\mathbf h),\beta(\mathbf h)\bigr)=1$. Similarly, the solutions of $Q({\bf{h}}) \ \Delta(Q({\bf{h}})) = 0$ are the matrices of the form
$$
Q({\bf{h}}) = \omega({\bf{h}})
\begin{bmatrix}
\gamma({\bf{h}}) \ \Delta^{-1}(\delta({\bf{h}})) & -\delta({\bf{h}})  \ \Delta^{-1}(\delta({\bf{h}})) \\
\gamma({\bf{h}})  \ \Delta^{-1}(\gamma({\bf{h}})) & -\delta({\bf{h}})  \ \Delta^{-1}(\gamma({\bf{h}}))
\end{bmatrix},
$$
where $\gamma({\bf{h}}), \delta({\bf{h}}),\omega({\bf{h}}) \in \mathcal{R}[{\bf{h}}]$ such that $\gcd(\gamma({\bf{h}}), \delta({\bf{h}})) = 1$.
\end{lemma}

\begin{proof}
  We solve $P(\mathbf h)\ \Delta^{-1}(P(\mathbf h))=0$, the other equation is analogous. If $P(\mathbf h)=0$, the claim is immediate by taking $\theta(\mathbf h)=0$. If $P({\bf{h}}) \neq 0$, there exist $\alpha(\mathbf h),\beta(\mathbf h)\in\mathcal R[\mathbf h]$, not both zero, with 
$$
\begin{bmatrix} \alpha({\bf{h}}) \\ \beta({\bf{h}}) \end{bmatrix} \in \ker\left(P({\bf{h}})\right).
$$
\np
If either $\alpha({\bf{h}}) = 0$ or $\beta({\bf{h}}) = 0$, then we deduce that
$$
P({\bf{h}}) =
\begin{bmatrix} 0 & \theta({\bf{h}}) \\ 0 & 0 \end{bmatrix}
\quad \text{or} \quad
P({\bf{h}}_m) =
\begin{bmatrix} 0 & 0 \\ \theta({\bf{h}}) & 0 \end{bmatrix},
$$
for some $\theta({\bf{h}}) \in \mathcal{R}[{\bf{h}}]$. If $\alpha({\bf{h}})\beta({\bf{h}}) \neq 0$, we may assume $\gcd(\alpha({\bf{h}}), \beta({\bf{h}})) = 1$. Since
$$
P({\bf{h}}) \begin{bmatrix} \alpha({\bf{h}}) \\ \beta({\bf{h}}) \end{bmatrix} = \begin{bmatrix} 0 \\ 0 \end{bmatrix},
$$
it follows that
$$
P({\bf{h}}) =
\begin{bmatrix}
\beta({\bf{h}}) x({\bf{h}}) & -\alpha({\bf{h}}) x({\bf{h}}) \\
\beta({\bf{h}}) y({\bf{h}}) & -\alpha({\bf{h}}) y({\bf{h}})
\end{bmatrix},
$$
for some $x({\bf{h}}), y({\bf{h}}) \in \mathcal{R}[{\bf{h}}]$.  
Moreover, the condition $P(\mathbf h)\ \Delta^{-1}(P(\mathbf h))=0$ yields
$$x({\bf{h}})\ \Delta^{-1}(\beta({\bf{h}}))\ \left(\beta({\bf{h}})\ \Delta^{-1}(x({\bf{h}})) - \alpha({\bf{h}}) \ \Delta^{-1}(y({\bf{h}}))\right) =0, $$
$$x({\bf{h}})\ \Delta^{-1}(\alpha({\bf{h}}))\ \left(\beta({\bf{h}}) \ \Delta^{-1}(x({\bf{h}})) - \alpha({\bf{h}}) \ \Delta^{-1}(y({\bf{h}}))\right) =0, $$
$$y({\bf{h}}) \ \Delta^{-1}(\beta({\bf{h}}))\ \left(\beta({\bf{h}})\ \Delta^{-1}(x({\bf{h}})) - \alpha({\bf{h}}) \ \Delta^{-1}(y({\bf{h}}))\right) =0, $$
$$y({\bf{h}})\ \Delta^{-1}(\alpha({\bf{h}}))\left(\beta({\bf{h}}) \ \Delta^{-1}(x({\bf{h}})) - \alpha({\bf{h}})\ \Delta^{-1}(y({\bf{h}}))\right) =0. $$
\np
\textbf{Case 1:} Suppose $\beta({\bf{h}}) \ \Delta^{-1}(x({\bf{h}})) - \alpha({\bf{h}})\ \Delta^{-1}(y({\bf{h}})) \neq 0$. Then the four relations above imply
$$
x({\bf{h}})\ \Delta^{-1}(\beta({\bf{h}})) =
x({\bf{h}})\ \Delta^{-1}(\alpha({\bf{h}})) =
y({\bf{h}})\ \Delta^{-1}(\beta({\bf{h}})) =
y({\bf{h}})\ \Delta^{-1}(\alpha({\bf{h}})) = 0.
$$
Hence $P(\mathbf h)=0$, which is covered with $\theta(\mathbf h)=0$.

\np
\textbf{Case 2:} Suppose $\beta({\bf{h}}) \ \Delta^{-1}(x({\bf{h}})) - \alpha({\bf{h}})\ \Delta^{-1}(y({\bf{h}})) = 0$.
Since $\gcd(\alpha({\bf{h}}), \beta({\bf{h}})) = 1$, it follows that
$$
x({\bf{h}}) = x'({\bf{h}}) \Delta(\alpha({\bf{h}}))
\quad \text{and} \quad
y({\bf{h}}) = y'({\bf{h}}) \Delta(\beta({\bf{h}})),
$$
for some $x'({\bf{h}}), y'({\bf{h}}) \in \mathcal{R}[{\bf{h}}]$.
Substituting back yields $x'({\bf{h}}) = y'({\bf{h}}) = \theta({\bf{h}})$, for some $\theta({\bf{h}}) \in \mathcal{R}[{\bf{h}}]$. Therefore,
$$
P({\bf{h}}) = \theta({\bf{h}})
\begin{bmatrix}
\beta({\bf{h}})\ \Delta(\alpha({\bf{h}})) & -\alpha({\bf{h}})\ \Delta(\alpha({\bf{h}})) \\
\beta({\bf{h}}) \ \Delta(\beta({\bf{h}})) & -\alpha({\bf{h}})\ \Delta(\beta({\bf{h}}))
\end{bmatrix},
$$
which proves the lemma.
\end{proof}

\np
\begin{proposition}\label{solutionsl(m|1)}
Let $\mathcal{R}$ be a unique factorization domain. Suppose $P({\bf h}),Q({\bf h})\in \Mat_2\big(\mathcal R[{\bf h}]\big)$ satisfy
$$
P(\mathbf h)\ \Delta^{-1}(P(\mathbf h)) = Q({\bf{h}}) \ \Delta(Q({\bf{h}})) = 0,
$$
$$
P({\bf{h}})\ \Delta^{-1}(Q({\bf{h}})) + Q({\bf{h}})\ \Delta(P({\bf{h}})) = a \I_2,
$$
where $a$ is an irreducible element in $\mathcal{R}$. Let $\GL_2(\mathcal{R}[{\bf{h}}])$ act on pairs $(P({\bf{h}}), Q({\bf{h}}))$ by
$$
G({\bf{h}}) \cdot \left(P({\bf{h}}), Q({\bf{h}})\right) := \left( G({\bf{h}})P({\bf{h}})\ \Delta^{-1}\left(G^{-1}({\bf{h}} )\right), \,\, G({\bf{h}})Q({\bf{h}})\ \Delta \left(G^{-1}({\bf{h}})\right) \right),
$$
where $G({\bf{h}}) \in \GL_2(\mathcal{R}[{\bf{h}}])$.
Then $(P({\bf{h}}), Q({\bf{h}}))$ lies in the same $\GL_2(\mathcal{R}[{\bf{h}}])$-orbit with 
$$
\left(
\begin{bmatrix} 0 & u({\bf{h}}) \\ 0 & 0 \end{bmatrix},
\begin{bmatrix} 0 & 0 \\ v({\bf{h}}) & 0 \end{bmatrix}
\right),
$$
where $u({\bf{h}}), v({\bf{h}}) \in \mathcal{R}[{\bf{h}}]$ satisfy $\ \Delta \left(u({\bf{h}})\right)v({\bf{h}}) = a$.
\end{proposition}

\begin{proof}
By Lemma~\ref{sl(m|1)reln}, there exist $\alpha(\mathbf h),\beta(\mathbf h),\theta(\mathbf h),
 \gamma(\mathbf h),\delta(\mathbf h),\omega(\mathbf h)\in \mathcal R[\mathbf h], $ with
$$
\gcd\!\big(\alpha(\mathbf h),\beta(\mathbf h)\big)
\;=\;\gcd\!\big(\gamma(\mathbf h),\delta(\mathbf h)\big)\;=\;1,
$$
such that
$$
P({\bf{h}}) =\theta(\mathbf h)\!
\begin{bmatrix}
\beta(\mathbf h)\ \Delta(\alpha(\mathbf h)) &
-\alpha(\mathbf h)\ \Delta(\alpha(\mathbf h))\\
\beta(\mathbf h)\ \Delta(\beta(\mathbf h)) &
-\alpha(\mathbf h)\ \Delta(\beta(\mathbf h))
\end{bmatrix},
$$
$$
Q({\bf{h}}) = \omega({\bf{h}})
\begin{bmatrix}
\gamma({\bf{h}}) \ \Delta^{-1}(\delta({\bf{h}})) & -\delta({\bf{h}})  \ \Delta^{-1}(\delta({\bf{h}})) \\
\gamma({\bf{h}})  \ \Delta^{-1}(\gamma({\bf{h}})) & -\delta({\bf{h}})  \ \Delta^{-1}(\gamma({\bf{h}}))
\end{bmatrix}.
$$
\np
Since $P({\bf{h}})\ \Delta^{-1}(Q({\bf{h}})) + Q({\bf{h}})\ \Delta(P({\bf{h}})) = a \I_2$, we deduce the following relations:
\begin{multline*}
    \Delta(\alpha({\bf{h}}))\ \Delta^{-1}(\delta({\bf{h}}))\
   \Bigl(-\beta({\bf{h}})\ \Delta^{-2}(\delta({\bf{h}}))\ \theta({\bf{h}}) \ \Delta^{-1}(\omega({\bf{h}})) +\\ +\alpha({\bf{h}}) \ \Delta^{-2}(\gamma({\bf{h}}))\ \theta({\bf{h}}) \ \Delta^{-1}(\omega({\bf{h}})) - \gamma({\bf{h}})\ \Delta^{2}(\alpha({\bf{h}}))\ \omega({\bf{h}})\ \Delta( \theta({\bf{h}}))  +\\
   +\delta({\bf{h}}) \ \Delta^{2}(\beta({\bf{h}}))\ \omega({\bf{h}}) \ \Delta( \theta({\bf{h}}))\Bigl)  = 0, 
\end{multline*}
\begin{multline*}
    \Delta(\beta({\bf{h}}))\ \Delta^{-1}(\gamma({\bf{h}}))
   \Bigl(-\beta({\bf{h}}_m)\ \Delta^{-2}(\delta({\bf{h}}))\ \theta({\bf{h}})\ \Delta^{-1}(\omega({\bf{h}})) + \\ +\alpha({\bf{h}})\ \Delta^{-2}(\gamma({\bf{h}}))\ \theta({\bf{h}}) \ \Delta^{-1}(\omega({\bf{h}}))- \gamma({\bf{h}})\ \Delta^{2}(\alpha({\bf{h}}))\ \omega({\bf{h}})\  \Delta(\theta({\bf{h}}))+\\ +\delta({\bf{h}})\ \Delta^{2}( \beta({\bf{h}}))\ \omega({\bf{h}}) \ \Delta(\theta({\bf{h}}))\Bigl)  = 0, 
\end{multline*}

\begin{multline}   \label{eq-sl(m|1)}
    \Bigl(\Delta(\alpha({\bf{h}}))\ \Delta^{-1}(\gamma({\bf{h}})) - \Delta^{-1}(\delta({\bf{h}}))\ \Delta(\beta({\bf{h}})) \Bigl)\Bigl(\beta({\bf{h}})\ \Delta^{-2}(\delta({\bf{h}}))\\
    -\alpha({\bf{h}}) \ \Delta^{-2}(\gamma({\bf{h}})) \Bigl)\theta({\bf{h}})\ \Delta^{-1}(\omega({\bf{h}})) = a  ,
\end{multline}     
For brevity, we denote
\begin{equation*} 
\begin{aligned}
t({\bf{h}}) := \; & 
- \beta({\bf{h}})\ \Delta^{-2}(\delta({\bf{h}}))\ \theta({\bf{h}}) \ \Delta^{-1}(\omega({\bf{h}}))
+ \alpha({\bf{h}})\ \Delta^{-2}(\gamma({\bf{h}}))\ \theta({\bf{h}})\ \Delta^{-1}(\omega({\bf{h}})) \\
& - \gamma({\bf{h}}) \ \Delta^2(\alpha({\bf{h}}))\ \omega({\bf{h}})\  \Delta(\theta({\bf{h}}))
+ \delta({\bf{h}})\ \Delta^{2}(\beta({\bf{h}}))\ \omega({\bf{h}})\ \Delta(\theta({\bf{h}})).
\end{aligned}
\end{equation*} 
\np
\textbf{Case 1:} Suppose $t({\bf{h}}) \neq 0$. Then $ \Delta(\alpha({\bf{h}}))\ \Delta^{-1}(\delta({\bf{h}})) =  \Delta(\beta({\bf{h}}))\ \Delta^{-1}(\gamma({\bf{h}})) =0$. 
It follows that $\alpha({\bf{h}}) = \gamma({\bf{h}}) = 0$, or $\beta({\bf{h}}) = \delta({\bf{h}}) = 0$. This settles Case~1.

\np
\textbf{Case 2:} Suppose $t({\bf{h}}) = 0$.  
From \eqref{eq-sl(m|1)}, we obtain:
\begin{equation}\label{eq2-sl(m|1)}
\begin{aligned}
&\bigl(\Delta(\alpha({\bf h}))\ \Delta^{-1}(\gamma({\bf h}))
      -\Delta^{-1}(\delta({\bf h}))\ \Delta(\beta({\bf h}))\bigr)\\
&\qquad\qquad \cdot\bigl(\delta({\bf h})\ \Delta^{2}(\beta({\bf h}))
      -\gamma({\bf h})\ \Delta^{2}(\alpha({\bf h}))\bigr)\,
      \omega({\bf h})\,\Delta(\theta({\bf h}))=a.
\end{aligned}
\end{equation}
\np
Since $a$ is an irreducible element of $\mathcal{R}$, \eqref{eq2-sl(m|1)} implies
$$
\Delta(\alpha({\bf h}))\ \Delta^{-1}(\gamma({\bf h}))
      -\Delta^{-1}(\delta({\bf h}))\ \Delta(\beta({\bf h})) = c,
$$
where $c \in \mathcal{R}^{\times}$. Hence
$$
G({\bf{h}}) :=
\begin{bmatrix}
\Delta^{-1}(\gamma({\bf{h}})) & -\Delta^{-1}(\delta({\bf{h}})) \\
\Delta(\beta({\bf{h}})) & -\Delta(\alpha({\bf{h}}))
\end{bmatrix}
\in \GL_2(\mathcal{R}[{\bf{h}}]).
$$

\np
A direct computation shows
$$G({\bf{h}})\ P({\bf{h}})\ \Delta^{-1}\left(G^{-1}({\bf{h}})\right) =\begin{bmatrix} 0 & c\theta({\bf{h}}) \\ 0 & 0 \end{bmatrix}, \quad G({\bf{h}})\ Q({\bf{h}})\ \Delta(G^{-1}({\bf{h}}))=\begin{bmatrix} 0 & 0 \\ -c\omega({\bf{h}}) & 0 \end{bmatrix}. $$

\np
Moreover, \eqref{eq2-sl(m|1)} also implies that $-c^2 \omega({\bf{h}})\ \Delta(\theta({\bf{h}})) = a$. By letting
$$
u({\bf{h}}) := c\theta({\bf{h}}), \qquad
v({\bf{h}}) := -c\omega({\bf{h}}),
$$
the Proposition follows.
\end{proof}

\begin{corollary}\label{genneralizedsl(1|1)}
Let $P({\bf{h}}), Q({\bf{h}}) \in \Mat_2(\mathbb{C}[{\bf{h}}])$ be such that
$$
P({\bf{h}})\ \sigma_m\Delta^{-1} (P({\bf{h}})) =
Q({\bf{h}})\ \sigma_m^{-1}\Delta(Q({\bf{h}})) = 0,
$$
$$
P({\bf{h}})\ \sigma_m\Delta^{-1}(Q({\bf{h}})) +
Q({\bf{h}})\ \sigma_m^{-1}\Delta(P({\bf{h}})) =
h_m\I_2.
$$
Then $(P({\bf{h}}), Q({\bf{h}}))$ is conjugate to
$$
\left( \begin{bmatrix} 0 & c h_{m} \\ 0 & 0 \end{bmatrix},
        \begin{bmatrix} 0 & 0 \\ c^{-1} & 0 \end{bmatrix} \right)
\quad \text{or} \quad
\left( \begin{bmatrix} 0 & c \\ 0 & 0 \end{bmatrix},
        \begin{bmatrix} 0 & 0 \\ c^{-1} h_{m} & 0 \end{bmatrix} \right),
\quad \text{for some } c \in \mathbb{C}^\times,
$$
up to the twisted conjugation: $(P'({\bf{h}}), Q'({\bf{h}})) \sim (P({\bf{h}}), Q({\bf{h}}))$ if and only if
$$
P'({\bf{h}}) = W^{-1}({\bf{h}})\, P({\bf{h}})\  \sigma_m\Delta^{-1}(W({\bf{h}})),\quad Q'({\bf{h}}) = W^{-1}({\bf{h}})\, Q({\bf{h}})\ \sigma_m^{-1}\Delta(W({\bf{h}})),
$$
for some $W({\bf{h}}) \in \GL_2(\mathbb{C}[{\bf{h}}])$.
\end{corollary}

    \begin{proof}
Set $\mathcal R:=\C[h_m]$. Then $\mathcal R[h_1,\dots,h_{m-1}]=\C[h_1,\dots,h_m]$, so the claim is a direct application of Proposition~\ref{solutionsl(m|1)}.
\end{proof}

\subsection{Classification of objects in the category $\mathcal{M}_{\mathfrak{sl}(m|1)}(2)$}
We now give a complete description of the category $\mathcal{M}_{\mathfrak{sl}(m|1)}(2)$, as stated in the following theorem.
\np
\begin{theorem} \label{classificationsl(m|1)}
Given ${\bf{a}} := (a_1,\dots, a_m) \in (\mathbb{C}^\times)^m$, $\mathcal{S} \subset {\bf{m}}$, define the matrices $E_{i\bar{1}}, E_{\bar{1}i}$ as follows
$$E_{i\bar{1}} = \begin{bmatrix} 0 & a_i h_i \\ 0 & 0  \end{bmatrix},\quad E_{\bar{1}i} = \begin{bmatrix} 0 & 0 \\ a_i^{-1} & 0  \end{bmatrix} \quad\text{if}\quad i \in \mathcal{S};$$
    $$E_{i\bar{1}} = \begin{bmatrix} 0 & a_i \\ 0 & 0  \end{bmatrix},\quad E_{\bar{1}i} = \begin{bmatrix} 0 & 0 \\ a_i^{-1} h_{i}
    & 0  \end{bmatrix} \quad\text{if}\quad i \not\in \mathcal{S} .$$
\np
Then these matrices $E_{i\bar{1}}, E_{\bar{1}i}$ determine a $\mathcal{U}(\mathfrak{sl}(m|1))$-module, denoted $M({\bf a},\mathcal S)$. Moreover, if $M\in\mathcal{M}_{\mathfrak{sl}(m|1)}(2)$, then $M \simeq M({\bf a},\mathcal S) $, for some ${\bf a}\in(\C^\times)^m$ and $\mathcal S\subset{\bf m}$.
\end{theorem}

\begin{proof}
Verifying that $M({\bf a},\mathcal S)$ is an $\mathcal{U}(\mathfrak{sl}(m|1))$-module is routine; we omit the details. Let $M\in\mathcal{M}_{\mathfrak{sl}(m|1)}(2)$, then $M = \C[\mathbf h]^{\oplus 2}$ as a vector space. With respect to the standard $\C[\mathbf h]$–basis, $e_{i\bar 1}$ and $e_{\bar 1 i}$ act by matrices $E_{i\bar 1},E_{\bar 1 i}\in\Mat_2(\C[\mathbf h])$ for all $i\in\mathbf m$. By Remark \ref{sl(m|1)-conjugacy}, we determine these matrices up to $\mathcal{M}_{\mathfrak{sl}(m|1)}(2)$–conjugation. For all $\mathbf {f}({\bf{h}}) \in \C[\mathbf h]^{\oplus 2}$,
$$
e_{m\bar{1}} \cdot e_{m\bar{1}} \cdot \mathbf {f}({\bf{h}}) = e_{\bar{1}m} \cdot e_{\bar{1}m} \cdot \mathbf {f}({\bf{h}}) = 0,\quad e_{m\bar{1}} \cdot e_{\bar{1}m} \cdot \mathbf {f}({\bf{h}}) + e_{\bar{1}m} \cdot e_{m\bar{1}} \cdot \mathbf {f}({\bf{h}}) = h_{\bar{1}}\cdot \mathbf {f}({\bf{h}}),
$$
it follows that the matrices $E_{m\bar 1},E_{\bar 1 m}$ satisfy
    $$E_{m\bar{1}}\ \Delta^{-1}_m(E_{m\bar{1}}) = E_{\bar{1}m}\ \Delta_m(E_{\bar{1}m}) = 0,$$
    $$E_{m\bar{1}}\ \Delta^{-1}_m(E_{\bar{1}m}) + E_{\bar{1}m}\ \Delta_m(E_{m\bar{1}}) = h_m\I_m.$$
\np
   By Corollary~\ref{genneralizedsl(1|1)}, up to $\mathcal{M}_{\mathfrak{sl}(m|1)}(2)$-conjugation, we have:
\begin{equation} \label{matrix-odd}
    E_{m\bar{1}}=\begin{bmatrix}0 & a_m h_m\\ 0 & 0\end{bmatrix},\quad E_{\bar{1}m}=\begin{bmatrix}0 & 0\\ a_m^{-1} & 0\end{bmatrix}\ \text{or}\ E_{m\bar{1}}=\begin{bmatrix}0 & a_m\\ 0 & 0\end{bmatrix},\quad E_{\bar{1}m}=\begin{bmatrix}0 & 0\\ a_m^{-1}h_m & 0\end{bmatrix},
\end{equation}
for some $a_m \in \mathbb{C}^\times$. For $i \in {\bf{m}} \setminus \{m\}$, let
$$
E_{i\bar{1}} =
\begin{bmatrix}
a_i({\bf{h}}) & b_i({\bf{h}}) \\
c_i({\bf{h}}) & d_i({\bf{h}})
\end{bmatrix},
\quad \text{where } a_i({\bf{h}}), b_i({\bf{h}}), c_i({\bf{h}}), d_i({\bf{h}}) \in \mathbb{C}[{\bf{h}}].
$$
\np
Since
$e_{i\bar{1}} \cdot e_{m\bar{1}} \cdot \mathbf {f}({\bf{h}}) = -e_{m\bar{1}} \cdot e_{i\bar{1}} \cdot \mathbf {f}({\bf{h}})$,  for all $\mathbf {f}({\bf{h}}) \in \C[\mathbf h]^{\oplus 2}$, it follows that
$$
\begin{bmatrix}
a_i({\bf{h}}) & b_i({\bf{h}}) \\
c_i({\bf{h}}) & d_i({\bf{h}})
\end{bmatrix}
\begin{bmatrix}
0 & \Delta_i^{-1} \left(u\left(h_{m}\right)\right) \\
0 & 0
\end{bmatrix}
= 
- \begin{bmatrix}
0 & u\left(h_{m}\right) \\
0 & 0
\end{bmatrix}
\Delta_m^{-1}
\left(
\begin{bmatrix}
a_i({\bf{h}}) & b_i({\bf{h}}) \\
c_i({\bf{h}}) & d_i({\bf{h}})
\end{bmatrix}
\right),
$$
where $u\left(h_m\right) \in \{a_m,\;a_m h_m\}$. This forces $c_i({\bf{h}}) = 0$.

\np
Moreover, since $e_{i\bar{1}} \cdot e_{i\bar{1}} \cdot \mathbf {f}({\bf{h}}) = 0$ for all $i \in {\bf{m}} \setminus \{m\}$ and $\mathbf {f}({\bf{h}}) \in \C[\mathbf h]^{\oplus 2}$, we have
$$
\begin{bmatrix}
a_i({\bf{h}}) & b_i({\bf{h}}) \\
0 & d_i({\bf{h}})
\end{bmatrix}
\Delta_i^{-1}
\left(
\begin{bmatrix}
a_i({\bf{h}}) & b_i({\bf{h}}) \\
0 & d_i({\bf{h}})
\end{bmatrix}
\right)
= 0,
$$
which yields $a_i({\bf{h}}) = d_i({\bf{h}}) = 0$, and $E_{i\bar{1}}=
\begin{bmatrix}
0 & b_i(\mathbf h)\\
0 & 0
\end{bmatrix}$, for all $i\in \mathbf m\setminus\{m\}.$

\np
By a similar argument, we deduce
$$
E_{\bar{1}i} =
\begin{bmatrix}
0 & 0 \\
e_i(\mathbf h) & 0
\end{bmatrix},
\qquad e_i(\mathbf h)\in \mathbb{C}[\mathbf h],\quad i\in \mathbf m\setminus\{m\}.
$$
\np
Since
$$
e_{i\bar{1}} \cdot e_{\bar{1}i} \cdot \mathbf {f}({\bf{h}}) +
e_{\bar{1}i} \cdot e_{i\bar{1}} \cdot \mathbf {f}({\bf{h}}) =
h_i\ \mathbf {f}({\bf{h}})
\quad \text{for all}\;\; i \in {\bf{m}} \setminus \{m\},\; \mathbf {f}({\bf{h}}) \in \C[\mathbf h]^{\oplus 2}, 
$$
we obtain
$$
\begin{bmatrix} 0 & b_i({\bf{h}}) \\ 0 & 0 \end{bmatrix}
\Delta_i^{-1} \left( \begin{bmatrix} 0 & 0 \\ e_i({\bf{h}}) & 0 \end{bmatrix} \right)
+
\begin{bmatrix} 0 & 0 \\ e_i({\bf{h}}) & 0 \end{bmatrix}
\Delta_i \left( \begin{bmatrix} 0 & b_i({\bf{h}}) \\ 0 & 0 \end{bmatrix} \right)
=h_i \I_2.
$$
Therefore,
$$
b_i({\bf{h}})\ \Delta_i^{-1}(e_i({\bf{h}})) =
e_i({\bf{h}})\ \Delta_i(b_i({\bf{h}})) =
h_i .
$$
\np
Consequently,
$$\bigl(b_i(\mathbf h),e_i(\mathbf h)\bigr)=\bigl(a_ih_i,\,a_i^{-1}\bigr)
\quad \text{or}\quad
\bigl(b_i(\mathbf h),e_i(\mathbf h)\bigr)=\bigl(a_i,\,a_i^{-1}h_i\bigr),
$$
for some $a_i\in\C^\times$. Equivalently,
\begin{equation} \label{matrix-odd2}
    E_{i\bar{1}}= \begin{bmatrix} 0 & a_ih_i \\ 0 & 0 \end{bmatrix},
    \quad
    E_{\bar{1}i} = \begin{bmatrix} 0 & 0 \\ a_i^{-1} & 0 \end{bmatrix}\;\;\; \text{or}\;\;\; E_{i\bar{1}}= \begin{bmatrix} 0 & a_i \\ 0 & 0 \end{bmatrix},
    \quad
    E_{\bar{1}i} = \begin{bmatrix} 0 & 0 \\ a_i^{-1}h_i & 0 \end{bmatrix}.
\end{equation}

\np
By \eqref{matrix-odd} and \eqref{matrix-odd2}, $M \simeq M({\bf a},\mathcal S) $, for some ${\bf a}=(a_1,\dots, a_m)\in(\C^\times)^m$ and $\mathcal S\subset{\bf m}$.
\end{proof}

\np
\begin{proposition} \label{isomthm(2)}
Let $\mathcal{S}_1, \mathcal{S}_2 \subset {\bf{m}}$ and ${\bf{a}}, {\bf{b}} \in (\mathbb{C}^\times)^m$.  
Then
$$
M({\bf a},\mathcal{S}_1)\;\simeq\; M({\bf b},\mathcal{S}_2)
\quad\Longleftrightarrow\quad
\mathcal{S}_1=\mathcal{S}_2 \ \text{ and }\ \exists\,\gamma\in\mathbb{C}^\times \ \text{such that}\ {\bf a}=\gamma\,{\bf b}.
$$
\end{proposition}

\begin{proof}
For $i\in{\bf m}$, denote by $E_{i\bar 1}^{({\bf a},\mathcal S_1)}, E_{\bar 1 i}^{({\bf a},\mathcal S_1)}$, 
the action-defining matrices of $M({\bf a},\mathcal S_1)$.
Similarly, write $
E_{i\bar 1}^{({\bf b},\mathcal S_2)}, E_{\bar 1 i}^{({\bf b},\mathcal S_2)}$, the action-defining matrices of $M({\bf b},\mathcal S_2)$. 

\np
The ``if'' direction is straightforward. For the ``only if'' direction, assume 
$M({\bf a},\mathcal S_1)\simeq M({\bf b},\mathcal S_2)$. 
We first show that this forces $\mathcal S_1=\mathcal S_2$. 
Suppose not; then, after relabelling if necessary, there exists $i\in \mathcal S_1\setminus \mathcal S_2$. Then,
$$E_{i\bar{1}}^{({\bf a},\mathcal S_1)} =
\begin{bmatrix} 0 & a_ih_i \\ 0 & 0 \end{bmatrix},\;\; E_{\bar{1}i}^{({\bf a},\mathcal S_1)} =
\begin{bmatrix} 0 & 0 \\ a_i^{-1} & 0 \end{bmatrix},\;\;
E_{i\bar{1}}^{({\bf b},\mathcal S_2)} =
\begin{bmatrix} 0 & b_i \\ 0 & 0 \end{bmatrix},\;\; E_{\bar{1}i}^{({\bf b},\mathcal S_2)} =
\begin{bmatrix} 0 & 0 \\ b_i^{-1}h_i & 0 \end{bmatrix}.$$

\np
Since $M({\bf a},\mathcal S_1)\;\simeq\; M({\bf b},\mathcal S_2)$, there exists
$$
W(\mathbf h)=
\begin{bmatrix}
w_1(\mathbf h) & w_2(\mathbf h)\\
w_3(\mathbf h) & w_4(\mathbf h)
\end{bmatrix}
\in \GL_2\big(\C[\mathbf h]\big)
$$
such that the following intertwining relations hold:
$$E_{i\bar 1}^{({\bf a},\mathcal S_1)}
=
W^{-1}(\mathbf h)\,E_{m\bar 1}^{({\bf b},\mathcal S_2)}\,
\Delta_i^{-1}\bigl(W(\mathbf h)\bigr),\quad E_{\bar 1 m}^{({\bf a},\mathcal S_1)}
=W^{-1}(\mathbf h)\,E_{\bar 1 m}^{({\bf b},\mathcal S_2)}\,
\Delta_i\bigl(W(\mathbf h)\bigr),$$
for all $i\in \mathbf m$. It follows that 
$$w_2({\bf{h}}) = w_3({\bf{h}}) = 0, \quad \text{and}\quad a_i h_i w_1({\bf{h}}) = b_i\ \Delta_i^{-1}\big(w_4({\bf{h}})\big).$$ 

\np
We obtain a contradiction since $w_1(h)w_4(h) \in \C^\times$. Therefore, 
$$M({\bf{a}},\mathcal{S}_1)\; \simeq \; M({\bf{b}},\mathcal{S}_2) \qquad\text{implies}\qquad \mathcal{S}_1 = \mathcal{S}_2.$$

\np
With $\mathcal{S}_1 = \mathcal{S}_2$, the above argument yields $w_1(\mathbf h),w_4(\mathbf h)\in\C^\times$, and $ a_i\,w_1=b_i\,w_4 $, for every $i\in{\bf m}$. Hence the ratio
$$
\gamma:=\frac{w_4}{w_1}=\frac{a_i}{b_i}\in\C^\times\quad\text{for all}\quad i\in{\bf m}.
$$
This completes the proof.
\end{proof}

\begin{proposition}\label{indecomM(a,S)}
For any ${\bf a}\in(\C^\times)^m$ and any subset $\mathcal S\subset{\bf m}$, the module $M({\bf a},\mathcal S)$ is indecomposable and of infinite length.
\end{proposition}

\begin{proof}
Let ${\bf a}\in(\C^\times)^m$ and $\mathcal S\subset{\bf m}$, we first prove that $M({\bf a},\mathcal S)$ is indecomposable. Let $\Phi \in 
\End_{\,\mathcal{U}(\mathfrak{sl}(m|1))}\bigl(M(\mathbf a,\mathcal S)\bigr)$. Then, there exists a matrix
$$
W_{\Phi}(\mathbf h)=
\begin{bmatrix}
w_1(\mathbf h) & w_2(\mathbf h)\\
w_3(\mathbf h) & w_4(\mathbf h)
\end{bmatrix}
\in \Mat_{2}\bigl(\C[\mathbf h]\bigr)
$$
such that $
\Phi\big(\mathbf {f}({\bf{h}})\big)
=
W_{\Phi}(\mathbf h)\,
\mathbf {f}({\bf{h}})
$, for all $\mathbf {f}({\bf{h}}) \in \C[\mathbf h]^{\oplus 2}$. Since
$$
\Phi\big(e_{i\bar{1}}\cdot \mathbf {f}({\bf{h}})\big)
   = e_{i\bar{1}}\cdot \Phi\big(\mathbf {f}({\bf{h}})\big),
\qquad
\Phi\big(e_{\bar{1}i}\cdot \mathbf {f}({\bf{h}})\big)
   = e_{\bar{1}i}\cdot \Phi\big(\mathbf {f}({\bf{h}})\big),
$$
for all  $i\in {\bf m}$, then $W_{\Phi}(\mathbf h)$ satisfies
$$W_{\Phi}({\bf{h}}) \begin{bmatrix} 0 & a_i h_i  \\ 0 & 0   \end{bmatrix} = \begin{bmatrix} 0 & a_i h_i \\ 0 & 0   \end{bmatrix}\ \Delta_i^{-1}\left( W_{\Phi}({\bf{h}})\right), \;\;\text{if}\;\; i \in \mathcal{S},$$
$$W_{\Phi}({\bf{h}})\begin{bmatrix} 0 & 0 \\ a_i^{-1} & 0   \end{bmatrix}= \begin{bmatrix} 0 & 0 \\ a_i^{-1} & 0   \end{bmatrix}\ \Delta_i\left( W_{\Phi}({\bf{h}})\right), \;\;\text{if}\;\; i \in \mathcal{S}, $$
$$W_{\Phi}({\bf{h}}) \begin{bmatrix} 0 & a_j \\ 0 & 0   \end{bmatrix} = \begin{bmatrix} 0 & a_j \\ 0 & 0   \end{bmatrix}\ \Delta_j^{-1}\left( W_{\Phi}({\bf{h}})\right), \;\;\text{if}\;\; j \not\in \mathcal{S},$$
$$W_{\Phi}({\bf{h}})\begin{bmatrix} 0 & 0 \\ a_j^{-1}h_j  & 0   \end{bmatrix}= \begin{bmatrix} 0 & 0 \\ a_j^{-1}h_j & 0   \end{bmatrix} \ \Delta_j\left( W_{\Phi}({\bf{h}})\right), \;\;\text{if}\;\; j \not\in \mathcal{S}. $$
\np
Consequently,
$$
W_{\Phi}(\mathbf h)
=\begin{bmatrix}
  w_1(\mathbf h) & 0\\
  0 & w_4(\mathbf h)
\end{bmatrix},
\qquad\text{where}\quad
w_1(\mathbf h)=\Delta_i^{-1}\big(w_4(\mathbf h)\big),
\quad
  \text{for all}\quad i\in \mathbf m.
$$

\np
Hence, there exists a polynomial $F(X)\in \C[X]$ such that
$$  W_{\Phi}(\mathbf h)
=\begin{bmatrix}
F\left(\sum_{j=1}^m h_j+m-1\right) & 0\\
0 & F\left(\sum_{j=1}^m h_j\right)
\end{bmatrix}$$

\np
Therefore, if $\Phi$ is an idempotent, then $\Phi = 0$ or $\Phi = \I_{2}$. So $M(\mathbf a,\mathcal S)$ is indecomposable.

\np 
 The preceding argument shows that every submodule $N\subseteq M(\mathbf a,\mathcal S)$ is of the form 
$$
N = M_F := 
F\left(\sum_{j=1}^m h_j+m-1\right) \mathbb{C}[{\bf{h}}]
\;\oplus\;
F\left(\sum_{j=1}^m h_j\right) \mathbb{C}[{\bf{h}}],
$$
for some $F(X)\in \C[X]$. Moreover, fixing an infinite sequence $\{\lambda_r\}_{r\geq 1}$ of complex numbers and setting
$$
F_0(X):=1, \qquad F_k(X):=\prod_{r=1}^k (X-\lambda_r),
$$
we obtain the filtration
$$
 \dots\; \subsetneq M_k \subsetneq \;\dots \;\subsetneq M_2 \subsetneq M_1 \subsetneq M_0 = M(\mathbf{a}, \mathcal{S}).
$$
where $M_k := 
\, F_k\left(\sum_{j=1}^m h_j+m-1\right)\ \mathbb{C}[{\bf{h}}]
\,\oplus\, 
 F_k\left(\sum_{j=1}^m h_j\right)\ \mathbb{C}[{\bf{h}}].$ 
\end{proof}

\subsection{Classification of objects in the categories $\mathcal{M}_{\mathfrak{sl}(m|1)}\bigl(1|1\bigr)$ and $\mathcal{M}^0_{\mathfrak{sl}(m|1)}\bigl(1|1\bigr)$ }
Each module $M({\bf a},\mathcal S)$ from Theorem~\ref{classificationsl(m|1)} carries a natural $\Z_2$-grading. Consequently, $M({\bf a},\mathcal S)$ is an object in both $\mathcal{M}_{\mathfrak{sl}(m|1)}\bigl(1|1\bigr)$ and $\mathcal{M}^0_{\mathfrak{sl}(m|1)}(1|1)$.

\np
\begin{definition}
Let ${\bf{a}} := (a_1, a_2, \dots, a_m) \in (\mathbb{C}^\times)^m$, $\mathcal{S} \subset {\bf{m}}$. Let $\overline{M({\bf{a}}, \mathcal{S})}$ to be the $\Z_2$-graded space $\C[{\bf{h}}] \oplus \C[{\bf{h}}]$, on which the $\mathcal{U}(\mathfrak{sl}(m|1))$ structure is given by the following action-defining matrices:
$$E_{i\bar{1}} = \begin{bmatrix} 0 & 0 \\ a_i h_i & 0  \end{bmatrix},\quad E_{\bar{1}m} = \begin{bmatrix} 0 & a_i^{-1} \\ 0 & 0  \end{bmatrix} \quad\text{if}\quad i \in \mathcal{S};$$
    $$E_{i\bar{1}} = \begin{bmatrix} 0 & 0 \\ 
   a_i & 0  \end{bmatrix} ,\quad E_{\bar{1}i} = \begin{bmatrix} 0 & a_i^{-1} h_i \\ 0 & 0  \end{bmatrix} \quad\text{if}\quad i \not\in \mathcal{S} .$$    
\end{definition}

\np
\begin{lemma} \label{simplifiedlem(1|1)}
For any ${\bf a}\in(\mathbb{C}^\times)^m$ and $\mathcal S \subset {\bf m}$, there is an isomorphism in
$\mathcal{M}_{\mathfrak{sl}(m|1)}(1|1)$:
$$
M({\bf{a}}, \mathcal{S}) \;\simeq\; \overline{M({\bf{a}}, \mathcal{S})}.
$$
\end{lemma}

\begin{proof}

The above isomorphism is given by the following odd map
  $$\Phi : M({\bf{a}}, \mathcal{S}) \;\to\; \overline{M({\bf{a}}, \mathcal{S})}$$
where 
$$\Phi\left(\begin{bmatrix} f_1({\bf{h}}) \\ f_2({\bf{h}}) \end{bmatrix} \right) = \begin{bmatrix} 0 & -1 \\ 1& 0 \end{bmatrix}\begin{bmatrix} f_1({\bf{h}}) \\ f_2({\bf{h}}) \end{bmatrix},\quad\text{for all}\quad f_i({\bf{h}}) \in \C[{\bf{h}}]. \qedhere $$
\end{proof}

\begin{remark} \label{nonisomorphic-paritypresev}
    It is clear that 
    $$
    M({\bf{a}}, \mathcal{S}_1) \;\not\simeq\; \overline{M({\bf{b}}, \mathcal{S}_2)} \quad \text{for all} \quad {\bf{a}}, {\bf{b}} \in (\mathbb{C}^\times)^m,\; \mathcal{S}_1, \mathcal{S}_2 \subset {\bf{m}},
    $$
    within the category $\mathcal{M}^0_{\mathfrak{sl}(m|1)}(1|1)$.
\end{remark}
\np
\begin{theorem} \label{classificationsl(m|1)-(1|1)-part1}
If $M \in \mathcal{M}_{\mathfrak{sl}(m|1)}\bigl(1|1\bigr)$, then $M \simeq M({\bf{a}}, \mathcal{S})$, for some ${\bf{a}} \in (\mathbb{C}^\times)^m$, $\mathcal{S} \subset {\bf{m}}$.
\end{theorem}

\begin{proof}
   Let $M \in \mathcal{M}_{\mathfrak{sl}(m|1)}\bigl(1|1\bigr)$. For $i\in \mathbf m$, let $E_{i\bar{1}}$ and $E_{\bar{1}i}$ be the action-defining matrices of $M$. By Lemma~\ref{structuresl(m|1)} and Remark~\ref{structure(1|1)}, for every $i\in \mathbf m$, we have
     $$E_{i\bar{1}} = \begin{bmatrix} 0 & a_{i\bar{1}}({\bf{h}}) \\ b_{i\bar{1}}({\bf{h}}) & 0 \end{bmatrix},\qquad E_{\bar{1}i} = \begin{bmatrix} 0 & a_{\bar{1}i}({\bf{h}}) \\ b_{\bar{1}i}({\bf{h}}) & 0 \end{bmatrix},$$
    with $a_{i\bar{1}}({\bf{h}}), a_{\bar{1}i}({\bf{h}}), b_{i\bar{1}}({\bf{h}}) , b_{\bar{1}i}({\bf{h}}) \in \C[{\bf{h}}]$. Since
$$ e_{i\bar{1}}\cdot e_{i\bar{1}} \cdot \mathbf {f}({\bf{h}})  \;=\; e_{\bar{1}i}\cdot e_{\bar{1}i} \cdot \mathbf {f}({\bf{h}}) \;=\; 0,\quad\text{for all}\quad \mathbf {f}({\bf{h}})\in\C[{\bf h}]^{\oplus 2},\; i\in{\bf m}, $$
it follows that
$$E_{i\bar{1}}\;\Delta_i^{-1}\bigl(E_{i\bar{1}}\bigr)=0,
\qquad
E_{\bar{1}i}\;\Delta_i\bigl(E_{\bar{1}i}\bigr)=0,\quad \text{for all}\;\; i\in \mathbf m,$$

\np
Hence, for all $i\in\mathbf m$,
$$
\begin{aligned}
& a_{i\bar1}(\mathbf h)\,b_{i\bar1}(\mathbf h)=0,
\qquad
\bigl(a_{i\bar1}(\mathbf h),\,b_{i\bar1}(\mathbf h)\bigr)\neq(0,0),\\[4pt]
& a_{\bar1 i}(\mathbf h)\,b_{\bar1 i}(\mathbf h)=0,
\qquad
\bigl(a_{\bar1 i}(\mathbf h),\,b_{\bar1 i}(\mathbf h)\bigr)\neq(0,0).
\end{aligned}
$$
\np
For distinct $i,j \in {\bf{m}}$, and all $\mathbf {f}({\bf{h}}) \in \C[{\bf{h}}]^{\oplus 2} $,
$$ e_{i\bar{1}}\cdot e_{j\bar{1}} \cdot \mathbf {f}({\bf{h}}) = -e_{j\bar{1}}\cdot e_{i\bar{1}} \cdot \mathbf {f}({\bf{h}}),\quad e_{\bar{1}i}\cdot e_{\bar{1}j} \cdot \mathbf {f}({\bf{h}}) = -e_{\bar{1}j}\cdot e_{\bar{1}i} \cdot \mathbf {f}({\bf{h}}).$$
\np
We deduce that, for each $i\in{\bf m}$,
$$E_{i\bar 1}=
\begin{bmatrix}0& a_{i\bar 1}(\mathbf h)\\0&0\end{bmatrix}
\;\;\text{or}\;\;E_{i\bar 1}=
\begin{bmatrix}0&0\\ b_{i\bar 1}(\mathbf h)&0\end{bmatrix}, \;and \;E_{\bar 1 i}=
\begin{bmatrix}0& a_{\bar 1 i}(\mathbf h)\\0&0\end{bmatrix}
\;\;\text{or}\;\;
E_{\bar 1 i}=
\begin{bmatrix}0&0\\ b_{\bar 1 i}(\mathbf h)&0\end{bmatrix}.$$
\np
Since,
$$
e_{i\bar{1}} \cdot e_{\bar{1}i} \cdot \mathbf {f}({\bf{h}}) +
e_{\bar{1}i} \cdot e_{i\bar{1}} \cdot \mathbf {f}({\bf{h}}) =
h_i \ \mathbf {f}({\bf{h}})
\quad \text{for all}\;\; i \in {\bf{m}},\ \mathbf f(\mathbf h)\in \C[\mathbf h]^{\oplus 2},
$$
we have the following two cases:
$$ E_{i\bar{1}} = \begin{bmatrix} 0 & a_{i\bar{1}}({\bf{h}}) \\ 0 & 0 \end{bmatrix}\quad \text{and} \quad  E_{\bar{1}i} = \begin{bmatrix} 0 & 0 \\ b_{\bar{1}i}({\bf{h}}) & 0 \end{bmatrix},\quad \text{for all} \quad i\in {\bf{m}},$$

or 
$$ E_{i\bar{1}} = \begin{bmatrix} 0 & 0 \\ b_{i\bar{1}}({\bf{h}}) & 0 \end{bmatrix}\quad \text{and} \quad  E_{\bar{1}i} =\begin{bmatrix} 0 & a_{\bar{1}i}({\bf{h}}) \\ 0 & 0 \end{bmatrix},\quad \text{for all} \quad i\in \bf{m}.$$
\np
If
$$
E_{i\bar{1}}=\begin{bmatrix}0& a_{i\bar{1}}(\mathbf h)\\0&0\end{bmatrix}
\quad\text{and}\quad
E_{\bar{1}i}=\begin{bmatrix}0&0\\ b_{\bar{1}i}(\mathbf h)&0\end{bmatrix}
\quad\text{for all}\quad i\in \mathbf m,
$$
then, proceeding as in the proof of Theorem~\ref{classificationsl(m|1)}, we deduce
$$
M \;\simeq\; M(\mathbf a,\mathcal S)
\quad\text{for some}\quad \mathbf a\in (\C^\times)^m,\ \mathcal S\subset \mathbf m.
$$

\np
If 
$$
E_{i\bar{1}}=\begin{bmatrix}0&0\\b_{i\bar{1}}(\mathbf h)&0\end{bmatrix}
\quad\text{and}\quad
E_{\bar{1}i}=\begin{bmatrix}0& a_{\bar{1}i}(\mathbf h)\\0&0\end{bmatrix}
\quad\text{for all}\;\; i\in \mathbf m,
$$
then an analogous argument yields
$$
M \;\simeq\; \overline{M(\mathbf a,\mathcal S)}
\quad\text{for some}\quad \mathbf a\in (\C^\times)^m,\ \mathcal S\subset \mathbf m.
$$
\np
By Lemma~\ref{simplifiedlem(1|1)}, the theorem follows.
\end{proof}

\begin{proposition}
    \label{classificationsl(m|1)-(1|1)-paritypres}
If $M \in \mathcal{M}^0_{\mathfrak{sl}(m|1)}\bigl(1|1\bigr)$, then 
$$M \;\simeq\; M({\bf{a}}, \mathcal{S})\qquad\text{or} \qquad M \simeq \overline{M({\bf{a}}, \mathcal{S})}, $$
for some ${\bf{a}}  \in (\mathbb{C}^\times)^m$, $\mathcal{S} \subset {\bf{m}}$.
\end{proposition} 

\begin{proof}
The claim follows from the proof of Theorem~\ref{classificationsl(m|1)-(1|1)-part1}, together with Remark~\ref{nonisomorphic-paritypresev}.
\end{proof}

\np
\begin{lemma} \label{even-hom}
    Let ${\bf{a}}, {\bf{b}} \in (\mathbb{C}^\times)^m$, $\mathcal{S}_1, \mathcal{S}_2  \subset {\bf{m}}$. Then every homomorphism 
    $$
\Phi \colon M({\bf a},\mathcal{S}_1)\;\to \;M({\bf b},\mathcal{S}_2),
$$
in the category 
    $\mathcal{M}_{\mathfrak{sl}(m|1)}(1|1)$ is purely even. More precisely, there exist
$w_1(\mathbf h),w_4(\mathbf h)\in \C[\mathbf h]$ such that, for all $\mathbf {f}({\bf{h}}) \in \C[\mathbf h]^{\oplus 2}$,
$$
\Phi\left(\mathbf {f}({\bf{h}})\right)
=
\begin{bmatrix}
w_1(\mathbf h) & 0\\
0 & w_4(\mathbf h)
\end{bmatrix}\
\mathbf {f}({\bf{h}}),
$$
with $w_1({\bf h})=\Delta_i^{-1}\bigl(w_4({\bf h})\bigr)$ for all $i\in{\bf m}$.
\end{lemma}

\begin{proof}
    Let  $\Phi \colon M({\bf a},\mathcal{S}_1)\;\to \;M({\bf b},\mathcal{S}_2)$ be a homomorphism in the category $\mathcal{M}_{\mathfrak{sl}(m|1)}(1|1)$. Then,
$$\Phi\left( \mathbf {f}({\bf{h}}) \right) \;=\; \begin{bmatrix}
w_1({\bf{h}}) & w_2({\bf{h}}) \\
w_3({\bf{h}}) & w_4({\bf{h}})
\end{bmatrix}\ \mathbf {f}({\bf{h}}),\quad \text{for all}\quad \mathbf {f}({\bf{h}}) \in \C[\mathbf h]^{\oplus 2},$$
where $w_i({\bf{h}}) \in \C[{\bf{h}}]$. Moreover,  $\Phi = \Phi_{\bar{0}} + \Phi_{\bar{1}}$ with
$$\Phi_{\bar{0}}\left(\mathbf {f}({\bf{h}})\right) = \begin{bmatrix}
w_1({\bf{h}}) & 0 \\
0 & w_4({\bf{h}})
\end{bmatrix}\ \mathbf {f}({\bf{h}}), \quad \Phi_{\bar{1}}\left(\mathbf {f}({\bf{h}})\right) = \begin{bmatrix}
0 & w_2({\bf{h}}) \\
w_3({\bf{h}}) & 0
\end{bmatrix}\ \mathbf {f}({\bf{h}}) .$$

\np
Since 
  $$\Phi\left(e_{i \bar{1}} \cdot\mathbf {f}({\bf{h}})\right) =\left(\Phi_{\bar{0}}+ \Phi_{\bar{1}}\right)\left(e_{i \bar{1}} \cdot\mathbf {f}({\bf{h}})\right)= e_{i\bar{1}} \cdot \Phi_{\bar{0}}\left(\mathbf {f}({\bf{h}})\right) - e_{i\bar{1}} \cdot \Phi_{\bar{1}}\left(\mathbf {f}({\bf{h}})\right),$$
  and
  $$\Phi\left(e_{\bar{1}i} \cdot\mathbf {f}({\bf{h}})\right) =\left(\Phi_{\bar{0}}+ \Phi_{\bar{1}}\right)\left(e_{ \bar{1}i} \cdot\mathbf {f}({\bf{h}})\right)= e_{\bar{1}i} \cdot \Phi_{\bar{0}}\left(\mathbf {f}({\bf{h}})\right) - e_{\bar{1}i} \cdot \Phi_{\bar{1}}\left(\mathbf {f}({\bf{h}})\right),$$
for all $\mathbf {f}({\bf{h}}) \in \C[\mathbf h]^{\oplus 2}$ and $i\in {\bf{m}}$, the lemma follows from direct computation.
\end{proof}

\np
\begin{remark}
From Lemma \ref{even-hom}, the statements of Propositions~\ref{isomthm(2)} and \ref{indecomM(a,S)} remain valid in the category $\mathcal{M}_{\mathfrak{sl}(m|1)}(1|1)$. The same proofs apply verbatim.
\end{remark}

 \section{The categories $\mathcal{M}_{\mathfrak{sl} (m|n)}(2)$ and $\mathcal{M}_{\mathfrak{sl} (m|n)}(1|1)$} \label{section5}
\np
In this section, we investigate the categories $\mathcal{M}_{\mathfrak{sl}(m|n)}(2)$ and $\mathcal{M}_{\mathfrak{sl} (m|n)}(1|1)$ with 
$m,n \geq 2$.  
Throughout, we fix $\C[\mathbf{h}] :=\C\!\bigl[h_{1},\dots,h_m,\,h_{\bar 1},\dots,h_{\overline{n-1}}\bigr].$ Note that one can extend Proposition~\ref{structuresl(m|1)} for objects of $\mathcal{M}_{\mathfrak{sl}(m|n)}(2)$. For our purposes, we record below only those actions needed for Theorem \ref{empty(m|n)}. For $i \in {\bf{m}}\setminus \{m\}$,
$$e_{i\bar{1}} \cdot \mathbf {f}({\bf{h}}) = E_{i\bar{1}}\ \sigma_i \sigma^{-1}_{\bar{1}}\left(\mathbf {f}({\bf{h}})\right),\quad e_{\bar{1}i} \cdot \mathbf {f}({\bf{h}}) = E_{\bar{1}i}\ \sigma_i^{-1} \sigma_{\bar{1}}\left(\mathbf {f}({\bf{h}})\right),$$
$$e_{i\bar{n}} \cdot \mathbf {f}({\bf{h}}) = E_{i\bar{1}}\ \Delta_i^{-1}\left(\mathbf {f}({\bf{h}})\right),\quad e_{\bar{n}i} \cdot \mathbf {f}({\bf{h}}) = E_{\bar{n}i}\ \Delta_i\left(\mathbf {f}({\bf{h}})\right),$$
$$e_{m\bar{1}} \cdot \mathbf {f}({\bf{h}}) = E_{m\bar{1}}\ \sigma_m \overline\Delta_{\bar{1}}\left(\mathbf {f}({\bf{h}})\right),\quad e_{\bar{1}m} \cdot \mathbf {f}({\bf{h}}) = E_{\bar{1}m}\  \sigma^{-1}_m \ \overline\Delta^{\,-1}_{\bar{1}}\left(\mathbf {f}({\bf{h}})\right) ,$$
$$e_{m\bar{n}} \cdot \mathbf {f}({\bf{h}}) = E_{m\bar{n}}\ \Delta^{-1}_m \overline\Delta\left(\mathbf {f}({\bf{h}})\right),\quad e_{\bar{n}m} \cdot \mathbf {f}({\bf{h}}) = E_{\bar{n}m}\  \Delta_m \overline\Delta^{\,-1}\left(\mathbf {f}({\bf{h}})\right) ,$$
where $\mathbf {f}({\bf{h}}) \in \C[\mathbf h]^{\oplus 2}$.
\begin{remark}
The notion of a $\mathcal{M}_{\mathfrak{sl}(m|1)}(2)$-conjugation extends to 
$\mathcal{M}_{\mathfrak{sl}(m|n)}(2)$. Namely, 
$$
\left(E_{i\bar{j}}, E_{\bar{j}i} \right)_{i \in {\bf{m}},\, \bar{j} \in {\bf{\bar{n}}}} \;
\sim_{\mathcal{M}_{\mathfrak{sl}(m|n)}(2)}\;
 \left(E'_{i\bar{j}}, E'_{\bar{j}i} \right)_{i \in {\bf{m}},\, \bar{j} \in {\bf{\bar{n}}}} 
$$
if and only if that there exists $W({\bf{h}}) \in \GL_2(\mathbb{C}[{\bf{h}}])$ satisfying relations similar to the ones in Definition
\ref{sl(m|1)_conjugation}.
An example of such a relation is
     $$E'_{m\bar{n}} = W^{-1}({\bf{h}})E_{m\bar{n}}\ \Delta^{-1}_m \overline\Delta \left(W({\bf{h}})\right),\qquad E'_{\bar{n}m} = W^{-1}({\bf{h}})E_{\bar{n}m}\ \Delta_m \overline\Delta^{\,-1} \left(W({\bf{h}})\right),$$
     see Section \ref{slbasis} for the definition of $\overline{\Delta}$.
 \end{remark}

\begin{lemma} \label{sl(m|n)}
    Let $P({\bf{h}}), Q({\bf{h}}) \in \Mat_2(\C[{\bf{h}}]) $ such that
    $$P({\bf{h}}) \ \Delta^{-1}_m \overline\Delta(P({\bf{h}})) \;=\; Q({\bf{h}}) \ \Delta_m \overline\Delta^{\,-1} (Q({\bf{h}})) \;=\; 0,$$
     $$P({\bf{h}}) \ \Delta^{-1}_m \overline\Delta (Q({\bf{h}})) +
    Q({\bf{h}}) \ \Delta_m \overline\Delta^{\,-1} (P({\bf{h}})) = h_m\I_2.$$
Then $(P({\bf{h}}), Q({\bf{h}}))$ are
$$\left( \begin{bmatrix} 0 & c h_m \\ 0 &  0 \end{bmatrix}, \begin{bmatrix} 0 & 0 \\ c^{-1} &  0 \end{bmatrix}  \right)~\text{or}~ \left( \begin{bmatrix} 0 & c \\ 0 &  0 \end{bmatrix}, \begin{bmatrix} 0 & 0 \\ c^{-1}h_m &  0 \end{bmatrix}  \right),\quad\text{for}\quad c\in \C^\times, $$
up to the twisted conjugation: $(P'({\bf{h}}), Q'({\bf{h}})) \sim (P({\bf{h}}), Q({\bf{h}}))$ if and only if
$$P'({\bf{h}}) = W^{-1}({\bf{h}})P({\bf{h}}) \ \Delta^{-1}_m \overline\Delta (W({\bf{h}}))  ,\qquad Q'({\bf{h}}) = W^{-1}({\bf{h}})Q({\bf{h}})\ \Delta_m \overline\Delta^{\,-1} (W({\bf{h}})),$$
for some $W({\bf{h}}) \in \GL_2(\C[{\bf{h}}])$.
    \end{lemma}

    \begin{proof}
        Similar to the proof of Corollary \ref{genneralizedsl(1|1)}.
    \end{proof}
    
    \np
\begin{theorem} \label{empty(m|n)}
For any integers $m,n \geq 2$, the category $\mathcal{M}_{\mathfrak{sl}(m|n)}(2)$ is empty.
\end{theorem}
\begin{proof}
Assume that there exists $M\in\mathcal{M}_{\mathfrak{sl}(m|n)}(2)$, then $M = \C[\mathbf h]^{\oplus 2}$ with $\bigl(E_{i\bar j},E_{\bar j i}\bigr)_{i\in{\bf m},\,\bar j\in{\bf\bar n}}$ are action-defining matrices of $M$. Since 
   $$e_{m\bar{n}} \cdot e_{m\bar{n}} \cdot \mathbf {f}({\bf{h}}) = e_{\bar{n}m}\cdot e_{\bar{n}m} \cdot \mathbf {f}({\bf{h}})=0, \quad e_{m\bar{n}} \cdot e_{\bar{n}m} \cdot \mathbf {f}({\bf{h}}) + e_{\bar{n}m} \cdot e_{m\bar{n}} \cdot \mathbf {f}({\bf{h}}) = h_m\cdot\mathbf {f}({\bf{h}}), $$
 for all $\mathbf {f}({\bf{h}}) \in \C[\mathbf h]^{\oplus 2}$, we obtain the following relations
$$E_{m\bar{n}}({\bf{h}}) \ \Delta^{-1}_m \overline\Delta(E_{m\bar{n}}({\bf{h}})) = E_{\bar{n}m}({\bf{h}}) \ \Delta_m \overline\Delta^{\,-1} (E_{\bar{n}m}({\bf{h}})) = 0,$$
$$E_{m\bar{n}}({\bf{h}})  \ \Delta^{-1}_m \overline\Delta (E_{\bar{n}m}({\bf{h}})) 
    +E_{\bar{n}m}({\bf{h}}) \ \Delta_m \overline\Delta^{\,-1} (E_{m\bar{n}}({\bf{h}})) = h_m\I_2. $$
 By Lemma~\ref{sl(m|n)}, up to $\mathcal{M}_{\mathfrak{sl}(m|n)}(2)$-conjugation, we deduce:
$$E_{m\bar{n}} = \begin{bmatrix} 0 & \alpha_{m\bar{n}} h_m \\ 0 & 0  \end{bmatrix},\; E_{\bar{n}m} = 
\begin{bmatrix} 0 & 0 \\ \alpha_{m\bar{n}}^{-1} & 0 \end{bmatrix},\quad \text{or} \quad E_{m\bar{n}} = \begin{bmatrix} 0 & \alpha_{m\bar{n}} \\ 0 & 0  \end{bmatrix},\;\; E_{\bar{n}m} = \begin{bmatrix} 0 & 0\\ \alpha_{m\bar{n}}^{-1} h_m & 0 \end{bmatrix},$$
where $\alpha_{m\bar{n}} \in \C^{\times}$. By the same argument as in Theorem~\ref{classificationsl(m|1)}, we deduce
$$E_{i\bar{n}} = \begin{bmatrix} 0 & \alpha_{i\bar{n}} h_i \\ 0 & 0  \end{bmatrix},\; E_{\bar{n}i} = 
\begin{bmatrix} 0 & 0 \\ \alpha_{i\bar{n}}^{-1} & 0 \end{bmatrix},\quad \text{or} \quad E_{i\bar{n}} = \begin{bmatrix} 0 & \alpha_{i\bar{n}} \\ 0 & 0  \end{bmatrix},\;\; E_{\bar{n}i} = \begin{bmatrix} 0 & 0\\ \alpha_{i\bar{n}}^{-1} h_i & 0 \end{bmatrix},$$
$$E_{m\bar{1}} = \begin{bmatrix} 0 & \alpha_{m\bar{1}} h_{\bar{1}} \\ 0 & 0  \end{bmatrix},\; E_{\bar{1}m} = 
\begin{bmatrix} 0 & 0 \\ \alpha_{m\bar{1}}^{-1} & 0 \end{bmatrix},\quad \text{or} \quad E_{m\bar{1}} = \begin{bmatrix} 0 & \alpha_{m\bar{1}} \\ 0 & 0  \end{bmatrix},\;\; E_{\bar{1}m} = \begin{bmatrix} 0 & 0\\ \alpha_{m\bar{1}}^{-1} h_{\bar{1}} & 0 \end{bmatrix},$$
and
$$E_{i\bar{1}} = \begin{bmatrix} 0 & \alpha_{i\bar{1}} (h_i+h_{\bar{1}} -h_m) \\ 0 & 0  \end{bmatrix},\qquad E_{\bar{1}i} = 
\begin{bmatrix} 0 & 0 \\ \alpha_{i\bar{1}}^{-1} & 0 \end{bmatrix},$$
$$\text{or} \quad E_{i\bar{1}} = \begin{bmatrix} 0 & \alpha_{i\bar{1}} \\ 0 & 0  \end{bmatrix},\qquad E_{\bar{1}i} = \begin{bmatrix} 0 & 0\\ \alpha_{i\bar{1}}^{-1} (h_i+h_{\bar{1}} -h_m) & 0 \end{bmatrix},$$
where $i \in \mathbf{m} \setminus\{m\}$, $\alpha_{i\bar{1}}, \alpha_{m\bar{1}}, \alpha_{i\bar{n}}, \alpha_{m\bar{n}} \in \C^\times$. Since
 $$e_{im} \cdot \mathbf {f}({\bf{h}}) = e_{i\bar{j}} \cdot e_{\bar{j}m}\cdot \mathbf {f}({\bf{h}}) + e_{\bar{j}m} \cdot e_{i\bar{j}} \cdot \mathbf {f}({\bf{h}}),\quad \text{for all}\quad \mathbf {f}({\bf{h}}) \in \C[\mathbf h]^{\oplus 2}, \bar{j} \in \bar{\mathbf n}.$$
We deduce
$$E_{im} = E_{i\bar{1}}\ \sigma_i \sigma^{-1}_{\bar{1}}\left(E_{\bar{1}m}\right) + E_{\bar{1}m}\  \sigma^{-1}_m \ \overline\Delta^{\,-1}_{\bar{1}} \left(E_{i\bar{j}} \right)= E_{i\bar{n}}\ \Delta_i^{-1}\left(E_{\bar{n}m} \right) + E_{\bar{n}m}\ \Delta_m \overline\Delta^{\,-1}\left(E_{i\bar{n}}\right).$$
\np
Hence
$$
E_{im}
=\frac{\alpha_{i\bar 1}}{\alpha_{m\bar 1}}\;\I_2
=\frac{\alpha_{i\bar n}}{\alpha_{m\bar n}}\;\I_2,
$$
and,
$$
E_{i\bar{1}} =
\begin{bmatrix}
0 & \alpha_{i\bar{1}} \\
0 & 0
\end{bmatrix}, \quad
E_{i\bar{n}} =
\begin{bmatrix}
0 & \alpha_{i\bar{n}} \\
0 & 0
\end{bmatrix}, \quad
E_{\bar{1}m} =
\begin{bmatrix}
0 & 0 \\
\alpha_{m\bar{1}}^{-1} & 0
\end{bmatrix}, \quad
E_{\bar{n}m} =
\begin{bmatrix}
0 & 0 \\
\alpha_{m\bar{n}}^{-1} & 0
\end{bmatrix}.
$$
\np
Therefore,
$$ E_{\bar{1}i} = \begin{bmatrix} 0 & 0 \\ \alpha_{i\bar{1}}^{-1}(h_i+h_{\bar{1}}-h_{m}) & 0  \end{bmatrix},\qquad E_{\bar{n}i} = \begin{bmatrix} 0 & 0 \\ \alpha_{i\bar{n}}^{-1}h_i & 0  \end{bmatrix},    $$
$$E_{m\bar{1}} = \begin{bmatrix} 0 & \alpha_{m\bar{1}}h_{\bar{1}} \\ 0 & 0  \end{bmatrix}, \qquad E_{m\bar{n}} = \begin{bmatrix} 0 & \alpha_{m\bar{n}}h_{m} \\ 0 & 0  \end{bmatrix} .$$
Since 
$$e_{mi} \cdot \mathbf {f}({\bf{h}}) = 
e_{m\bar{j}} \cdot e_{\bar{j}i} \cdot \mathbf {f}({\bf{h}}) 
+ e_{\bar{j}i} \cdot e_{m\bar{j}} \cdot \mathbf {f}({\bf{h}}), \quad\text{for all}\quad \mathbf {f}({\bf{h}})\in \C[\mathbf h]^{\oplus 2},\bar{j} \in \bar{\mathbf n}.$$ 
it follows that
\begin{multline*}
  E_{mi} = E_{m\bar{1}}\ \sigma_m \overline\Delta_{\bar{1}}\left(E_{\bar{1}i} \right) + E_{\bar{1}i}\ \sigma^{-1}_i\sigma_{\bar{1}} \left(E_{m\bar{1}} \right)\\ =\frac{\alpha_{m\bar{1}}}{\alpha_{i\bar{1}}}\begin{bmatrix} h_{\bar{1}}(h_i+h_{\bar{1}}-h_{m}+1) & 0 \\ 0 & (h_{\bar{1}}-1)(h_i+h_{\bar{1}}-h_{m})
\end{bmatrix}  
\end{multline*}
and
$$E_{mi} = E_{m\bar{n}}\ \Delta^{-1}_m \overline\Delta\left(E_{\bar{n}i}\right) + E_{\bar{n}i}\ \Delta_i\left(E_{m\bar{n}}\right) = \frac{\alpha_{m\bar{n}}}{\alpha_{i\bar{n}}} \begin{bmatrix} (h_i+1)h_{m} & 0 \\ 0 & h_i(h_{m}-1)
\end{bmatrix}$$

\np
This yields a contradiction, and the theorem follows.
\end{proof}

\begin{remark}
For all integers $m,n\geq 2$, the categories $\mathcal{M}_{\mathfrak{sl}(m|n)}(1|1)$ and $\mathcal{M}^0_{\mathfrak{sl}(m|n)}(1|1)$ are also empty.
\end{remark}

\np
\textbf{Acknowledgment.} We would like to thank Charles Paquette and David Wehlau for many valuable discussions and insightful suggestions. All authors were partially supported by the Natural Sciences and Engineering Research Council of Canada.

\end{document}